\documentclass[11pt]{amsart}

\RequirePackage{amsthm}
\usepackage{amssymb}
\usepackage{marvosym}
\RequirePackage{array}
\RequirePackage{enumerate}
\RequirePackage{cancel}
\RequirePackage{bbm}
\usepackage{longtable}
\usepackage[all]{xy}

\usepackage{graphicx}
\usepackage{tikzexternal}
\newtheoremstyle%
{Theorem}%
{}%
{}%
{\itshape}%
{}%
{}%
{.}%
{ }%
{\thmname{\bfseries #1}%
\thmnumber{\;\bfseries #2}%
\thmnote{\;(\bfseries #3)}}%
\theoremstyle{plain}
\newtheorem{thm}{Theorem}[section]
\newtheorem{cor}[thm]{Corollary}
\newtheorem{lemma}[thm]{Lemma}
\newtheorem{prop}[thm]{Proposition}
\theoremstyle{definition}
\newtheorem{dfn}[thm]{Definition}
\newtheorem{ex}[thm]{Example}

\newtheorem{remark}[thm]{Remark}
\newtheorem*{theorem*1}{Theorem 1}
\newtheorem*{theorem*2}{Theorem 2}
\newtheorem*{theorem*3}{Theorem 3}
\numberwithin{figure}{section}
\numberwithin{table}{section}
\numberwithin{equation}{section}

\newenvironment{enum}
	{\begin{enumerate} [\upshape (1)]}
	{\end{enumerate}}
\newenvironment{enumi}
	{\begin{enumerate} [\upshape (i)]}
	{\end{enumerate}}
\newenvironment{enuma}
	{\begin{enumerate} [\upshape (a)]}
	{\end{enumerate}}

\newcommand{\calc}{\mathcal{C}}

\newcommand{\cale}{\mathcal{E}}
\newcommand{\calf}{\mathcal{F}}

\newcommand{\calh}{\mathcal{H}}

\newcommand{\calr}{\mathcal{R}}

\DeclareMathOperator{\Ext}{Ext}

\DeclareMathOperator{\End}{End}

\DeclareMathOperator{\Hom}{Hom}

\DeclareMathOperator{\cmod}{\mathsf{mod}}
\DeclareMathOperator{\Fan}{Fan}
\DeclareMathOperator{\AG}{AG}

\DeclareMathOperator{\gdim}{gl. dim}

\usepackage{tipa} 

\newcommand{\fieldstyle}[1]{\mathbb{#1}}
\renewcommand{\AA}{\fieldstyle{A}}

\newcommand{\NN}{\fieldstyle{N}}

\newcommand{\e}{\varepsilon}
\renewcommand{\tilde}{\widetilde}

\newcommand{\dimvv}[3]{\begin{smallmatrix}
						#1\\ #2\\ #3
					   \end{smallmatrix}}

\newcommand{\dimv}[2]{\begin{smallmatrix}
						#1\\ #2
					   \end{smallmatrix}}
\newcommand{\ds}{\ \phantom{1}}
\newcommand{\ms}[1]{\ #1\ }

\newcommand{\orientationofthearrow}{counter-clockwise }

\newcommand{\ot}{\leftarrow}

\newcommand{\za}{\alpha}
\newcommand{\zb}{\beta}
\newcommand{\zd}{\delta}
\newcommand{\zD}{\Delta}

\newcommand{\zg}{\gamma}
\newcommand{\zG}{\Gamma}

\title{Algebras from surfaces without punctures}
\author{Lucas David-Roesler}
\author{Ralf Schiffler}\thanks{The second author was supported by the NSF grant  DMS-1001637 and by the University of Connecticut.}

\newcommand{\res}{\mathit{res}}
\newcommand{\cut}[3]{\chi_{#1,#2,#3}}
\newcommand{\dd}{\dagger}
\renewcommand{\dag}{^\dd}
\newcommand{\quasitri}{quasi-triangle}

\newcommand{\df}[1]{\emph{#1}}
\newcommand{\IF}{\text{ if }}

\begin{document}
\begin{abstract}
We introduce a new class of finite dimensional gentle algebras, the \emph{surface algebras}, which  are constructed from an unpunctured Riemann surface with boundary and marked points by introducing cuts in  internal triangles of an arbitrary triangulation of the surface. We show that surface algebras are endomorphism algebras of partial cluster-tilting objects in generalized cluster categories, we compute the invariant of Avella-Alaminos and Geiss for surface algebras and we provide a geometric model for the module category of  surface algebras.
\end{abstract}
\maketitle
\section{Introduction}

We introduce a new class of finite dimensional gentle algebras, the \emph{surface algebras}, which includes the hereditary, the tilted, and the cluster-tilted algebras of Dynkin type $\mathbb{A}$ and Euclidean type $\tilde{\mathbb{A}}$.
These algebras are constructed from an unpunctured Riemann surface with boundary and marked points by introducing cuts in  internal triangles of an arbitrary triangulation of the surface. 

To be more precise, let $T$ be a triangulation of a bordered unpunctured Riemann surface $S$ with a set of marked points $M$, and let $(Q_T,I_T)$ be the bound quiver associated to $T$ as in \cite{CCS,ABCP}. The corresponding algebra $B_T=kQ_T/I_T$, over  an algebraically closed field $k$, is a finite dimensional gentle algebra \cite{ABCP}. Moreover, $B_T$  is the endomorphism algebra of the cluster-tilting object corresponding to $T$ in the generalized cluster category associated to $(S,M)$, see \cite{CCS,BMRRT,Amiot,BZ}. Each internal triangle in the triangulation $T$ corresponds to an oriented $3$-cycle in the quiver $Q_T$, and the relations for the algebra $B_T$ state precisely that the composition
of any two arrows in an oriented $3$-cycle is zero in $B_T$.

If the surface is a disc or an annulus then the corresponding cluster algebra, as defined in \cite{FST}, is acyclic, and, in this case, the algebra $B_T$  is cluster-tilted of type $\mathbb{A}$, if $S$ is a disc; and of type $\tilde{\mathbb{A}}$, if $S$ is an annulus \cite{CCS,BMR}. It has been shown in \cite{ABS} that every cluster-tilted algebra is the trivial extension of a tilted algebra $C$ by the canonical $C$-bimodule $\Ext^2_C(DC,C)$, where $D$ denotes the standard duality.
The quiver of the tilted algebra $C$ contains no oriented cycle and can be obtained from the quiver of $B_T$ by an admissible cut, that is, by deleting one arrow in every oriented $3$-cycle. 
Moreover, it has been shown in \cite{BFPPT} that an algebra is iterated tilted of  Dynkin type  $\AA$  of global
dimension at most two if and only if it is the quotient of a cluster-tilted algebra of the same
 type by an admissible cut.
It is then natural to ask, what kind of algebras we can get from admissible cuts of algebras $B_T$ coming from other surfaces. 

This motivates the definition of a \emph{surface algebra}, which is constructed by cutting a triangulation $T$ of a surface $(S,M)$ at internal triangles. Cutting an internal triangle $\triangle$ means replacing the triangle $\triangle$ by a quadrilateral $\triangle^\dd$ with one side on the boundary of the same surface $S$ with an enlarged set of marked points $M^\dd$, see Definition \ref{def cut}. Cutting as many internal triangles as we please, we obtain a partial triangulation $T^\dd$ of a surface with marked points $(S,M^\dd)$, to which we can associate an algebra $B_{T^\dd}=kQ_{T^\dd}/I_{T^\dd}$ in a very similar way to the construction of $B_T$ from $T$, see Definition \ref{def surface algebra}. This algebra $B_{T^\dd}$ is called a \emph{surface algebra of type} $(S,M)$.
{A surface algebra is called an \emph{admissible cut} if it is obtained by cutting every internal triangle exactly once.}

Our first main results are the following.
\begin{theorem*1} 
\it Every surface algebra is isomorphic to the endomorphism algebra of a partial cluster-tilting object in a generalized cluster category. More precisely, if the surface algebra $B_{T^\dd} $ is given by the 
cut $(S,M^\dd,T^\dd)$ of the triangulated surface $(S,M,T)$, then  \[B_{T^\dd}\cong \End_{\calc_{(S,M^\dd)}}T^\dd,\]
where $T^\dd$ denotes the object in cluster category $\calc_{(S,M^\dd)}$ corresponding to $T^\dd$. 

\end{theorem*1}

\begin{theorem*2}
\it
If $(S,M^\dd,T^\dd)$ is an admissible cut of $(S,M,T)$ then
\begin{enuma}
\item  $Q_{T^\dd}$ is an admissible cut of $Q_T$.
\item  $B_{T^\dd}$ is of global dimension at most two.
\item The tensor algebra of $B_{T^\dd}$ with respect to the $B_{T^\dd}$-bimodule 
		\[\Ext^2_{B_{T^\dd}}(DB_{T^\dd},B_{T^\dd})\]
	 is isomorphic to the   algebra $B_T$. 
\end{enuma}
\end{theorem*2}

{Part (c) of Theorem 2 implies that the cluster category associated to $B_{T^\dd}$ is the same as the cluster category associated to the surface $(S,M)$. Therefore, the surface algebras $B_{T^\dd}$ which are admissible cuts are algebras of cluster type $(S,M)$ in the sense of \cite{AO}.  In \cite{AO}, the authors study algebras of global dimension two whose cluster type is acyclic, which, in our setting, corresponds to admissible cuts from a disc or an annulus. }

Applying the result of \cite{BFPPT} mentioned above, we see that the admissible cut surface algebras  of the disc with $n+3$ marked points are precisely the 
iterated tilted algebras of type $\AA_n$ whose global dimension is at most two.
{For all surfaces other than the disc, we show that the surface algebras form several different classes under derived equivalence. For admissible cuts from the annulus, this recovers a result in \cite{AO}. }

To show that the surface algebras are not always derived equivalent, we use an invariant introduced by Avella-Alaminos and Geiss in \cite{AG}, which we call the AG-invariant for short. For each surface algebra $B_{T^\dd}$, we compute the AG-invariant in terms of the original triangulated surface $(S,M,T)$ and the number of cuts on each boundary component, see Theorem \ref{thm AG calc}. In particular, we show that already for the annulus there are different cuts coming from the same triangulation $T$ such that the corresponding surface algebras have different  AG invariants and hence are not derived equivalent. 

We want to point out here that  Ladkani has very recently also computed the AG invariant for the algebras $B_T$ (without cuts) in the special case where each boundary component has exactly one marked point, and used it to classify the surfaces $(S,M)$ such that any two triangulations $T_1,T_2$ of $(S,M)$ give rise to derived equivalent algebras $B_{T_1}$ and $B_{T_2}$, see \cite{Lad}.
Let us also mention that the AG-invariant has been used by Bobinski and Buan in \cite{BB} to classify algebras that are derived equivalent to cluster-tilted algebras of type $\AA$ and $\tilde{\AA}$.

We then study the module categories of the surface algebras. Since surface algebras are gentle, their indecomposable modules are either string modules or band modules. In the special case where the surface is not cut, the indecomposable modules and  the irreducible morphisms in the module category of the algebras $B_T$  have been described by Br\"ustle and Zhang in \cite{BZ}  in terms of the of generalized arcs on the surface $(S,M,T)$. One of the main tools used in \cite{BZ} is the description of irreducible morphisms between string modules by Butler and Ringel \cite{BR}.  We generalize the results of Br\"ustle and Zhang  to the case of arbitrary surface algebras $B_{T^\dd}$, and we describe the indecomposable modules in terms of certain permissible generalized arcs in the surface
$(S,M^\dd,T^\dd)$ and the irreducible morphisms in terms of  pivots of these arcs in the surface. In this way, we construct a combinatorial category $\cale^\dd$ of permissible generalized arcs in $(S,M^\dd,T^\dd)$ and define a functor  \[H\colon \cale^\dd\to \cmod B_{T^\dd}.\] This construction is inspired by the construction of the cluster category of type $\AA$ as a category of diagonals in a convex polygon by Caldero, Chapoton and the second author in \cite{CCS}. We then show the following theorem.

\begin{theorem*3}\
\begin{enuma}
\item The functor $H$ is faithful and induces a dense, faithful functor from $\cale\dag$ to
the category of string modules over $B_{T\dag}$. Moreover, $H$ maps irreducible morphisms to irreducible morphisms and commutes with Auslander-Reiten translations.
\item If the surface $S$ is a disc, then $H$ is an equivalence of categories.
\item If the algebra $B_{T\dag}$ is of finite representation type, then $H$ is an equivalence of categories.
\end{enuma}\end{theorem*3}

As an application of our results, we provide a geometric model for the module category of any iterated tilted algebra of type $\AA$ of global dimension two in terms of permissible diagonals in a partially triangulated polygon.

The paper is organized as follows.  In section \ref{sect 2}, we recall definitions and results that we need at a later stage. Section \ref{sect 3} is devoted to the definition of surface algebras and their fundamental properties. The computation of the AG-invariant is contained in section \ref{sect 4} and, in section \ref{sect 5}, we study the module categories of surface algebras in terms of arcs in the surface. The application to iterated tilted algebras of type $\AA$ is presented in section \ref{sect 6}.

\section{Preliminaries and notation}\label{sect 2}
In this section, we recall concepts that we need and fix notations.

\subsection{Gentle algebras, string modules and band modules}
Recall from \cite{AS} that a finite-dimensional algebra $B$ is {\em gentle} if it admits a
presentation $B=kQ/I$ satisfying the following {\nobreak conditions:}
\begin{itemize}
\item[(G1)] At each point of $Q$ start at most two arrows and stop at
  most two arrows.
\item[(G2)] The ideal $I$ is generated by paths of length 2.
\item[(G3)] For each arrow $ b$ there is at most one arrow $ a$
  and at most one arrow $ c$ such that $ a  b \in I$
  and $ b  c \in I$.
\item[(G4)] For each arrow $ b$ there is at most one arrow $ a$
  and at most one arrow $ c$ such that $ a  b \not\in I$
  and $ b  c \not\in I$.
\end{itemize}

An algebra $B=kQ/I$ where $I$ is generated by paths and $(Q,I)$ satisfies the two
conditions (G1) and (G4) is called a \df{string algebra} (see \cite{BR}), thus every
gentle algebra is a string algebra.
Butler and Ringel have shown in \cite{BR} that, for a finite dimensional string algebra, there are two types of indecomposable modules, the string modules and the band modules. 
We recall the definitions here. 

For any arrow $ b \in Q_1$, we denote by $ b^{-1}$ a \emph{formal inverse} for $ b$, with $s( b^{-1})=t( b)$, $t( b^{-1})=s( b)$, and we set $( b^{-1})^{-1} =  b$. 

		A \emph{walk} of length $n \geq 1$ in $Q$ is a sequence $w=a_1 \cdots a_n$ where each $a_i$ is an arrow or a formal inverse of an arrow and such that $t(a_i) = s(a_{i+1})$, for any $i \in \{1, \ldots, n-1\}$. The \emph{source of the walk} $w$ is $s(w)=s(a_1)$ and the \emph{target of the walk} $w$ is $t(w) = t(a_n)$. We define a \emph{walk $e_i$ of length zero} for any point $i \in Q_0$ such that $s(e_i)=t(e_i)=i$. 
		
				If $(Q,I)$ is a bound quiver, a \emph{string} in $(Q,I)$ is either a walk of length zero or a walk $w=a_1 \cdots a_n$ of length $n \geq 1$ such that $a_i \neq a_{i+1}^{-1}$ for any $i \in \{1, \ldots, n-1\}$ and such that no walk of the form $a_i a_{i+1} \cdots a_t$ nor its inverse belongs to $I$ for $1 \leq i$ and $t \leq n$.  A \emph{band} is a string $b=a_1\cdots a_n$ such that $s(b)=t(b)$, and any power of $b$ is a string, but $b$ itself is not a proper power of a string.

Any string $w$ gives rise to a string module $M(w)$ over $B$, whose underlying vector space consists of the direct sum of one copy of the field $k$ for each vertex in the string $w$, and the action of an arrow $a$ on $M(w)$ is induced by the identity morphism $1 \colon k\to k$ between the copies of $k$ at the endpoints of $a$, if $a$ or $a^{-1}$ is in $w$, and zero otherwise. Each band $b=a_1a_2\cdots a_n$ defines a family of band modules $M(b,\ell,\phi)$, where $\ell$ is a positive integer and $\phi$ is an automorphism of $k^\ell$. The underlying vector space of $M(b,\ell,\phi)$ is the direct sum of one copy of $k^\ell$ for each vertex in $b$, and the action of an arrow $a$ is induced by the identity morphism $1\colon k^\ell\to k^\ell$, if $a=a_1,a_2,\ldots,a_{n-1}$, by the automorphism $\phi$ if $a=a_n$ and the action is zero if $a$ is not in $w$.

\subsection{The AG-invariant}
First we recall from \cite{AG} the combinatorial definition of the derived invariant of
Avella-Alaminos and Geiss. We will refer to this as the AG-invariant. Throughout let $A$ be a gentle
$k$-algebra with bound quiver $(Q,I)$, $Q=(Q_0,Q_1,s,t)$ where $s,t\colon Q_1\to Q_0$ are the source and
target functions on the arrows.

\begin{dfn}
A \df{permitted path} of $A$ is a path $C=a_1a_2\cdots a_n$ which is not in $I$. We
say a permitted path is a \df{non-trivial permitted thread} of $A$ if for all arrows $ b\in Q_1$,
neither $ b C$ or $C b$ is not a permitted path. These are the `maximal' permitted paths of
$A$. Dual to this, we define the \df{forbidden paths} of $A$ to be a sequence 
$F= a_1a_2\cdots a_n$ such that  $a_i\ne a_j$ unless $i=j$, and 
$a_ia_{i+1}\in I$, for $i=1,\dots,n-1$. A forbidden path $F$ is a \df{non-trivial forbidden thread} 
if for all $ b\in Q_1$, neither $ b F$ or $F b$ is a forbidden path.
We also require \df{trivial permitted} and \df{trivial forbidden threads}. Let $v\in Q_0$ such that
there is at most one arrow starting at $v$ and at most one arrow ending at $v$.
Then
the constant path $e_v$ is a trivial permitted thread if when there are arrows $ b, c\in
Q_1$ such that $s( c)=v=t( b)$ then $ b c\not\in I$. Similarly, $e_v$ is a trivial forbidden
thread if when there are arrows $ b, c\in Q_1$ such that $s( c)=v=t( b)$ then 
$ b c\in I$.
  
Let $\calh$ denote the set of all permitted threads and $\calf$ denote the set of all forbidden threads.
\end{dfn}

Notice that each arrow in $Q_1$, is both a permitted and a forbidden
path. Moreover, the constant path at each sink and at each source will 
simultaneously satisfy the definition for a permitted and a forbidden 
thread because there are no paths going through $v$.

We fix a choice of functions $\sigma,\e\colon Q_1\to \{-1,1\}$ characterized by the following conditions.
\begin{enumerate}
\item If $ b_1\neq  b_2$ are arrows with $s( b_1)=s( b_2)$, then 
	$\sigma( b_1)=-\sigma( b_2)$.
\item If $ b_1\neq  b_2$ are arrows with $t( b_1)=t( b_2)$, then $\e( b_1)=-\e( b_2)$.
\item If $ b, c$ are arrows with $s( c)=t( b)$ and $ b c\not\in I$, 
then $\sigma( c)=-\e( b)$.
\end{enumerate}
Note that the functions need not be unique. Given a pair $\sigma$ and $\e$, we can define
another pair $\sigma':=-1\sigma$ and $\e':=-1\e$.

These functions naturally extend to paths in $Q$. Let 
$C= a_1a_{2}\cdots a_{n-1}a_n$ be a path. Then 
$\sigma(C) = \sigma(a_1)$ and $\e(C)=\e(a_n)$. We can also extend these functions 
to trivial threads. Let $x,y$ be vertices in $Q_0$, $h_x$ the trivial permitted thread at $x$, and
$p_y$ the trivial forbidden thread at $y$. Then we set
\begin{align*} 
\sigma(h_x) = -\e(h_x) &= -\sigma(a), & \IF s(a)&=x, \text{ or}\\
 \sigma(h_x) = -\e(h_x) &= -\e( b), & \IF t( b)&=x
 \end{align*}
and
\begin{align*}
\sigma(p_y) = \e(p_y)& = -\sigma( c), &  \IF s( c) &= y, \text{ or}\\
\sigma(p_y) =\e(p_y)&  = -\e(d), & \IF t(d) &= y ,
\end{align*} 
where $a, b, c,d\in Q_1$. Recall that these arrows are unique if they exist. 

\begin{dfn}
The AG-invariant $\AG(A)$ is defined to be a function depending on the ordered pairs generated 
by the following algorithm.
\begin{enumerate}
\item \begin{enumerate}
	\item Begin with a permitted thread of $A$, call it $H_0$.
	\item \label{alg:HtoF} To $H_i$ we associate $F_i$, the forbidden thread which ends at 
		$t(H_i)$ and such that  $\e(H_i)=-\e(F_i)$. Define $\varphi(H_i) := F_i$.
	\item \label{alg:FtoH}To $F_i$ we associate $H_{i+1}$, the permitted thread which starts 
		at $s(F_i)$ and such  that $\sigma(F_i)=-\sigma(H_{i+1})$. Define $\psi(F_i):= H_{i+1}$.
	\item Stop when $H_n=H_0$ for some natural number $n$.  Define $m=\sum_{i=1}^n \ell(F_i)$, 
		where $\ell(C)$  is the length (number of arrows) of a path $C$. In this way we obtain the 
		pair $(n,m)$. 
	\end{enumerate}
\item Repeat (1) until all permitted threads of $A$ have occurred.
\item For each oriented cycle in which each pair of consecutive arrows form a relation, we associate 
the ordered pair $(0,n)$, where $n$ is the length of the cycle. 
\end{enumerate}
We define $\AG(A)\colon \NN^2\to \NN$ where $\AG(A)(n,m)$ is the number of times the ordered pair 
$(n,m)$ is formed by the above algorithm. 
\end{dfn}

The algorithm defining $\AG(A)$ can be thought of as dictating a walk in the quiver $Q$, where we
move forward on permitted threads and backward on forbidden threads.

\begin{remark}\label{rmk:bijection}
Note that the steps \eqref{alg:HtoF} and  \eqref{alg:FtoH} of this algorithm give two different bijections $\varphi$ and $\psi$ between the set of permitted
threads $\calh$  and the set of forbidden threads which do not start and end in 
the same vertex. 
We will often refer to the permitted (respectively forbidden) thread ``corresponding'' to a given 
forbidden (respectively permitted) thread. This correspondence is referring to the bijection $\varphi$ (respectively $\psi$). 
\end{remark}

\begin{thm}\label{thm AG}
\begin{enuma}
\item Any two derived equivalent gentle algebras have the same AG-invariant.
\item Gentle algebras which have at most one (possibly non-oriented) cycle in their quiver are derived equivalent if and only if they have the same AG-invariant.
\end{enuma}
\end{thm}
\begin{proof}
See \cite[Theorems A and C]{AG}.
\end{proof}

\subsection{Surfaces and triangulations}\label{sect surfaces} 

In this section, we recall a construction of \cite{FST} in the
case of surfaces without punctures.

Let $S$ be a connected oriented unpunctured Riemann surface with boundary $\partial S$ and let $M$ be a
non-empty finite subset of the boundary $\partial S$. The elements of $M$ are called \emph{marked
points}. We will refer to the pair $(S,M)$ simply by \emph{unpunctured surface}. The orientation of
the surface will play a crucial role.

We say that two curves in $S$ \emph{do not cross} if they do not intersect
each other except that endpoints may coincide.

\begin{dfn}\label{def arc}
An \emph{arc} $\zg$ in $(S,M)$ is a curve in $S$ such that 
\begin{itemize}
\item[(a)] the endpoints are in $M$,
\item[(b)] $\zg$ does not cross itself,
\item[(c)] the relative interior of $\zg$ is disjoint from $M$ and
  from the boundary of $S$,
\item[(d)] $\zg$ does not cut out a monogon or a digon. 
\end{itemize}   
\end{dfn}     

\begin{dfn}
A \emph{generalized arc} is a curve in $S$ which satisfies the conditions (a), (c) and (d)
of Definition \ref{def arc}.
\end{dfn}

 Curves that connect two
marked points and lie entirely on the boundary of $S$ without passing
through a third marked point are called \emph{boundary segments}.
Hence an arc is a curve between two marked points, which does not
intersect itself nor the boundary except possibly at its endpoints and
which is not homotopic to a point or a boundary segment.

Each generalized arc is considered up to isotopy inside the class of such curves. Moreover, each 
generalized arc is considered up to orientation, so if a generalized arc has endpoints $a,b\in M$ then it can
 be represented by a curve that runs from $a$ to $b$, as well as by a curve that runs from $b$ to $a$.

For any two arcs $\zg,\zg'$ in $S$, let $e(\zg,\zg')$ be the minimal number of crossings of
$\zg$ and $\zg'$, that is, $e(\zg,\zg')$ is the minimum of the numbers of crossings of
curves $\za$ and $\za'$, where $\za$ is isotopic to $\zg$ and $\za'$ is isotopic to $\zg'$.
Two arcs $\zg,\zg'$ are called \emph{non-crossing} if $e(\zg,\zg')=0$. A
\emph{triangulation} is a maximal collection of non-crossing arcs.
The arcs of a triangulation cut the surface into
\emph{triangles}. Since $(S,M)$ is an unpunctured surface, the three sides of each triangle
are distinct (in contrast to the case of surfaces with punctures). A triangle in $T$ is
called an \emph{internal triangle} if none of its sides is a boundary segment. We often
refer to the triple $(S,M,T)$ as a \emph{triangulated surface}.

Any two triangulations of $(S,M)$ have
the same number of elements, namely \[n=6g+3b+|M|-6,\]  where $g$ is the
genus of $S$, $b$ is the number of boundary components and $|M|$ is the
number of marked points. The number $n$ is called the \emph{rank} of $(S,M)$.
Moreover, any two triangulations of $(S,M)$ have
the same number of triangles, namely\[ n-2(g-1)-b.
\]
Note that $b> 0$ since the set $M$ is not empty. We do not allow $n$ to be negative or zero, so we
have to exclude the cases where $(S,M)$ is a disc with one, two or three marked points. Table
\ref{table 1} gives some examples of unpunctured surfaces.

\begin{table}
\begin{center}
  \begin{tabular}{ c | c | c || l  }
  \  $b$\ \  &\ \ $ g$ \ \   & \ \ $ \vert M\vert$  \ \  &\  surface \\ \hline
    $1$ & 0 & $n+3$ & \ disc \\ 
    1 & 1 & $n-3$ & \ torus with disc removed \\
    1 & 2 & $n-9 $& \ genus 2 surface with disc removed \\\hline 
    2 & 0 & $n$ & \ annulus\\
    2 & 1 & $n-6$ & \ torus with 2 discs removed \\ 
    2 & 2 & $n-12$ & \ genus 2 surface with 2 discs removed \\ \hline
    3 & 0 & $n-3$ & \ pair of pants \\ 
  \end{tabular}
  \medskip
\end{center}
\caption{Examples of unpunctured surfaces}\label{table 1}
\end{table}

\subsection{Jacobian algebras from surfaces}
If $T=\{\tau_1,\tau_2,\ldots,\tau_n\}$ is a triangulation of an unpunctured surface $(S,M)$, we define a quiver $Q_T$ as follows.
$Q_T$ has $n$ vertices, one for each arc in $T$. We will denote the vertex corresponding to $\tau_i$
simply by $i$. The number of arrows from $i$ to $j$ is the number of triangles $\triangle$ in $T$ such
that the arcs $\tau_i,\tau_j$ form two sides of $\triangle$, with $\tau_j$ following $\tau_i$ when going
around the triangle $\triangle$ in the \orientationofthearrow orientation, see Figure \ref{fig quiver}
for an example. Note that the interior triangles in $T$ correspond to oriented 3-cycles in $Q_T$.
 
\begin{figure}
\centering
\begin{tikzpicture}[scale=.55]
{ []
\draw (0,0) circle[radius=5cm] ;
\filldraw[fill=black!20] (1.5,1.5) circle[radius=1cm] ;
\filldraw[fill=black!20] (-1.5,-1.5) circle[radius=1cm] ;

\foreach \x [count=\n] in {0,1,...,11}{
	\node[name=p\n] at ($(1.5,1.5)+(30*\x:1cm)$) {};
	\node[name=n\n] at ($(-1.5,-1.5)+(30*\x:1cm)$) {};
	\node[name=o\n] at (30*\x:5cm) {};
}

\node[solid] at (90:5cm) {};
\node[solid] at (180:5cm) {};
\node[solid] at ($(-1.5,-1.5)+(60:1cm)$) {};
\node[solid] at ($(1.5,1.5)+(180:1cm)$) {};
\node[solid] at ($(1.5,1.5)+(-30:1cm)$) {};

\draw (o7.center) -- (p7.center) node[int] {$\tau_1$} 
	(o7.center) ..controls +(0:2cm) and +(135:1cm) .. (n3) node[int] {$\tau_2$} 
	(o7.center) .. controls +(-75:2cm) and +(170:1cm) .. ($(o9)!.75!(n10)$) node[int] {$\tau_3$}
	 ($(o9)!.75!(n10)$) .. controls +(-10:.5cm) and +(200:.5cm) .. ($(n12)!.3!(o10)$) 
	 ($(n12)!.3!(o10)$) .. controls +(20:1cm) and (-70:1cm) .. (n3.center)
	 (p12.center) .. controls +(60:1cm) and +(-60:.5cm) .. ($(p4)!.33!(o3)$) node[int,pos=.9] {$\tau_7$}
	 ($(p4)!.33!(o3)$) .. controls +(130:.cm) and +(-40:1cm) .. (o4.center) 
	 (n3.center) .. controls +(0:3cm) and +(-90:1cm) .. ($(p1)!.7!(o1)$) node[int] {$\tau_4$}
	 ($(p1)!.7!(o1)$) .. controls +(90:1cm) and +(-20:1cm) .. (o4.center) 
	 (o4.center) -- (p7.center) node[int] {$\tau_8$} 
	 (p7.center) -- (n3.center)  node[int] {$\tau_6$}
	  (p12.center) .. controls +(-110:1cm) and +(30:1cm) .. (n3.center)  node[int] {$\tau_5$}; 
}
{[xshift=7cm,scale=1.5]
	\node[name=1] at (1,1) {1};
	\node[name=2] at (0,0) {2};
	\node[name=3] at (1,-1) {3};
	\node[name=4] at (5,-1) {4};
	\node[name=5] at (3.5,0) {5};
	\node[name=6] at (2,0) {6};
	\node[name=7] at (5,1) {7};
	\node[name=8] at (3,1.5) {8};
	\draw[->] (1) edge (2) edge (8)
		(2) edge (6) edge (3)
		(6) edge (1) edge (5) 
		(5) edge (4)
		(4) edge (7) edge (3)
		(7) edge (5) edge (8);
}
\end{tikzpicture}
\caption{A triangulation and its quiver} \label{fig quiver}
\end{figure}
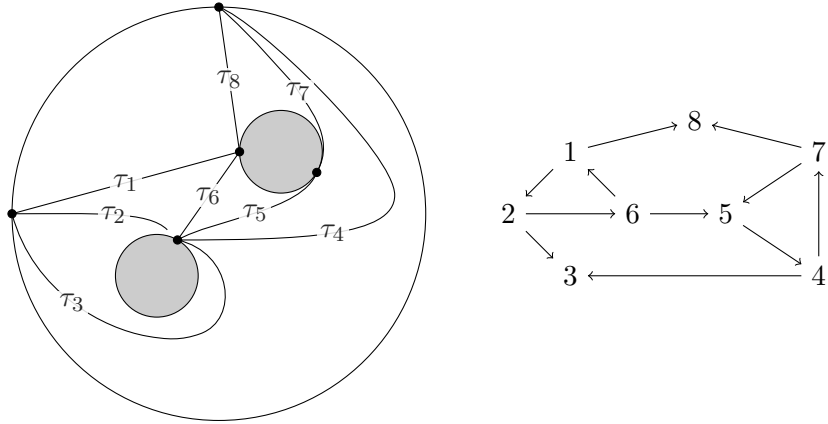   

In \cite{FST}, the authors associate a cluster algebra $\mathcal{A}(Q_T)$ to this quiver; the
cluster algebras obtained in this way are called cluster algebras from (unpunctured) surfaces and
have been studied in \cite{FST,FT,S3,MSW, FeShTu}, and the corresponding cluster categories in
\cite{CCS, BZ}.

Following \cite{ABCP,LF}, let $W$ be the sum of all oriented 3-cycles in $Q_T$. Then $W$ is a
potential, in the sense of \cite{DWZ}, which gives rise to to a Jacobian algebra $B_T=
\textup{Jac}(Q_T,W)$, which is defined as the quotient of the path algebra of the quiver $Q_T$ by
the two-sided ideal generated by the subpaths of length two of each oriented 3-cycle in $Q_T$.

\begin{prop}\label{prop 2.7} $B_T$ is a gentle algebra. \end{prop}
\begin{proof}
This is shown in \cite{ABCP}.
\end{proof}

\subsection{Cuts and admissible cuts} Let $Q$ be a quiver. An oriented cycle $C$ in $Q$ is called \emph{chordless} if $C$ is a full subquiver of $Q$ and has the property that for every vertex $v$ in $C$, there is exactly one arrow in $C$ starting and one arrow in $C$ ending at $v$. We define a \emph{cut} of $Q$ to be a
subset of the set of arrows of $Q$ with the property that each arrow in the cut lies in an oriented
chordless cycle in $Q$. Following \cite{Fernandez}, such a cut is called \emph{admissible} if it
contains exactly one arrow of each oriented chordless cycle of $Q$.

Now let $C=kQ/I$ be a quotient of the path algebra of $Q$ by an admissible ideal $I$. Then an
algebra is said to be obtained from $C$ by a cut if it is isomorphic to a quotient $kQ/\langle I\cup
\zG \rangle$, where $\zG$ is a cut of $Q$.

\begin{prop} \label{prop 2.8} Any algebra obtained by a cut from a gentle algebra is gentle. \end{prop}

\begin{proof} If a bound quiver satisfies the conditions (G1)--(G4) of the definition of a
gentle algebra, then deleting an arrow will produce a bound quiver that still satisfies
these conditions.
\end{proof}

\section{Algebras from surfaces without punctures} \label{sect 3}
In this section, we introduce the surface algebras and develop their fundamental properties.

Let $(S,M)$ be a surface without punctures, $T$ a
triangulation, $Q_T$ the corresponding quiver, and $B_T$ the Jacobian algebra. Throughout this
section, we assume that, if $S$ is a disc, then $M$ has at least $5$ marked points, thus we exclude
the disc with $4$ marked points. Recall that the oriented $3$-cycles in the quiver $Q_T$ are in
bijection with the interior triangles in the triangulation $T$.

\subsection{Cuts of triangulated surfaces}
We want to define a geometric object which corresponds to cuts of the quiver $Q_T$, and, to that
purpose, we modify the set $M$ to a new set of marked points $M^\dagger$, and we modify the
triangulation $T$ to a partial triangulation $T^\dagger$ of the surface $(S,M^\dagger)$.\footnote{We use the dagger symbol $\dd$ to  indicate the cut.}

We need some terminology. Each arc has two ends defined by trisecting the arc and deleting the
middle piece. If $\triangle$ is a triangle with sides $\za,\zb,\zg$, then the six ends of the three sides
can be matched into three pairs such that each pair forms one of the angles of the triangle $\triangle$.

Let $\triangle$ be an internal triangle of $T$ and let $v\in M$ be one of its vertices. Let $\zd'$ and
$\zd''$ be two curves on the boundary, both starting at $v$ but going in opposite directions, and
denote by $v'$ and $v''$ there respective endpoints. Moreover, choose $\zd',\zd''$ short enough 
such that $v'$ and $v''$ are not in $M$, and no point of $M$ other that $v$ lies on the curves
$\zd',\zd''$.  We can think of $v',v''$ being obtained by moving the point $v$ a small amount in 
either direction along the boundary, see Figure \ref{fig cut def} for an example.
Define \[\chi_{v,\triangle} (M) =\left(M\setminus\{v\}\cup\{v',v''\}\right).\]
\begin{figure}
\begin{tikzpicture}[scale=1.5,lbl/.style={fill=white,opacity=.66,shape=circle,inner sep=1,outer sep =2}]
{[]
\shade[shading=axis,bottom color=black!20, top color=black!5]  (0,0) rectangle (3,.25);
\draw (0,0) -- (3,0) node[solid,pos=.5,scale=1,name=v] {};
\node[above,lbl] at (v) {$v$};
\draw (v.center) -- ++(-60:2cm) node[int,pos=.8] {$\beta$}
	-- ++(180:2cm) node[above,pos=.66] {$\triangle$}
	-- (v.center) node[int,pos=.2] {$\alpha$};
\draw[line width=1.5pt]
	(v.center) -- +(-60:.75cm) node[right,pos=.4] {$\bar\beta$}
	(v.center) -- +(-120:.75) node[left,pos=.4] {$\bar\alpha$}
	(v.center) -- +(180:.75cm) node[above,lbl,pos=.8] {$\delta'$}
	(v.center) -- +(0:.75) node[above,lbl,pos=.8] {$\delta''$};
}
\draw[->] (3,-.86) -- (4,-.86) node[above,pos=.5] {$\cut{v}{\alpha}{\beta}$};
{[xshift=4cm]
\shade[shading=axis,bottom color=black!20, top color=black!5]  (0,0) rectangle (3,.25);
\draw (0,0) -- (3,0) node[solid,pos=.25,scale=1,name=v'] {} node[solid,pos=.75,scale=1,name=v''] {};
\node[above,lbl] at (v') {$v'$};
\node[above,lbl] at (v'') {$v''$};
\draw  (v') -- ++(270:1.73cm) node[int] {$\alpha\dag$}
	  -- ++(0:1.5cm) node[above,pos=.25] {$\triangle\dag$}
	  -- (v'') node[int] {$\beta\dag$};
}
\end{tikzpicture}
\caption{A local cut at $v$ relative to $\alpha,\beta$. The internal triangle $\triangle$ in $T$ becomes a quasi-triangle $\triangle^\dd$ in $T^\dd$.}\label{fig cut def}
\end{figure}
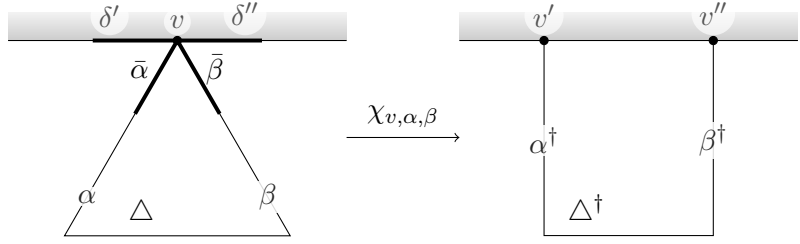

Let $\bar\za$ and $\bar \zb$ be ends of two sides $\za,\zb$ of $\triangle$ such that $\bar\za,\bar\zb$
form an angle of $\triangle$ at $v$. If $\bar\gamma $ is an end of an arc $\zg\in T$ such that $\bar\zg$
is incident to $v$, let $\bar\zg'$ be a curve in the interior of $S$ homotopic to
\[\Bigg\{\begin{array}{ll} 
\text{the concatenation of $\bar\zg$ and $\zd'$, if $\bar\za$ lies between $\bar \zg$ and $\bar\zb$, or $\bar \zg=\bar\za$;}\\ 
\text{the concatenation of $\bar\zg$ and $\zd''$, if $\bar\zb$ lies between $\bar \zg $ and $\bar\za$, or $\bar \zg=\bar\zb$.}\\ 
\end{array}\]
Then let $\chi_{v,\za,\zb}(\zg)$ be the arc obtained from $\zg$ by
replacing the end $\bar\zg$ by $\bar\zg'$. If both ends of $\zg$ are incident to $v$ then
$\chi_{v,\za,\zb}(\zg)$ is obtained from $\zg$ by replacing both ends with the appropriate new end; 
this is the case in the example in Figure \ref{fig cut}. 
If $\gamma\ne\alpha,\beta$ then, abusing notation, we will denote the arc $\chi_{v,\za,\zb}(\zg)$ again by $\gamma$. The arcs obtained from $\za$ and $\zb$ will be denoted by $\za^\dd=\chi_{v,\za,\zb}(\za)$ and $\zb^\dd=\chi_{v,\za,\zb}(\zb)$, see Figure \ref{fig cut def}.
 Define 
\[\chi_{v,\za,\zb}(T)=\left(T\setminus \{\zg\in T\mid \zg \textup{ incident to $v$}\}\right)
\cup \{\chi_{v,\za,\zb}(\zg)\mid \zg \textup{ incident to $v$}\} .\]

Finally, let $\chi_{v,\za,\zb}(S,M,T)=(S,\chi_{v,\zD}(M),\chi_{v,\za,\zb}(T))$. 
Let us point out that
$(S,\chi_{v,\triangle}(M))$ is a surface which has exactly one marked point more than $(S,M)$, and that
$\chi_{v,\za,\zb}(T)$ has the same number of arcs as $T$. Therefore $\chi_{v,\za,\zb}(T)$ is a
partial triangulation of the surface $(S,\chi_{v,\triangle}(M))$.   
We denote by $\triangle^\dd$ the quadrilateral with sides $\za^\dd,\zb^\dd,\zg$ and the new boundary segment between $v'$ and $v''$.

\begin{dfn} \label{def cut}\ 
\begin{enumerate} 
\item The partially triangulated surface $\chi_{v,\za,\zb}(S,M,T)$ is called the \emph{local cut} of
$(S,M,T)$ at $v$ relative to $\za,\zb$.
\item A \emph{cut} of the triangulated surface $(S,M,T)$ is a partially triangulated surface
$(S,M^\dagger,T^\dagger)$ obtained by applying a sequence of local cuts $\chi_{v_1,\za_1,\zb_1},
\ldots, \chi_{v_t,\za_t,\zb_t}$ to $(S,M,T)$, subject to the condition that the triangle $\triangle_i$ in
the $i$-th step is still an internal triangle after $i-1$ steps. 
\end{enumerate} 
\end{dfn} 
Thus  we are allowed to cut each internal triangle of $T$ at most once.
The quadrilaterals $\triangle^\dd_i$ in $T^\dd$ corresponding to $\triangle_i$ in $T$ are called \emph{quasi-triangles}. Note that a quasi-triangle is a quadrilateral that has exactly one side on the boundary.
  
\begin{dfn}
A cut of $(S,M,T)$ is called an \emph{admissible cut} if every internal triangle of $T$ is cut exactly once.
\end{dfn}
\begin{figure}
\centering
\begin{tikzpicture}[scale=.33]
{[]
	\draw (0,0) circle [radius=5cm] ;
	\filldraw[fill=black!20] (0,0) circle [radius=1.5cm] ;
	
	\foreach \x [count=\n] in {0,1,2,3}{
		\node[name=i\n] at (90*\x:1.5cm) {};
		\node[name=o\n] at (90*\x:5cm) {};
	}
	\node[solid] at (90:1.5cm) {};
	\node[solid,label=90:$v$] at (90:5cm) {};
	\node[solid] at (-90:5cm) {};
	
	\draw (o2.center) .. controls +(-135:2) and +(90:2)  .. ($(o3)!.25!(i3)$) node[int] {$\gamma$}
		 ($(o3)!.25!(i3)$) .. controls +(-90:2) and +(180:2) .. ($(o4)!.25!(i4)$)
		($(o4)!.25!(i4)$) .. controls +(0:2) and +(-90:2)  .. ($(o1)!.25!(i1)$)
		($(o1)!.25!(i1)$) .. controls +(90:2) and +(-45:2)  .. (o2.center)  
		(o2.center) -- (i2.center) node[int] {$\alpha$} 
		(i2.center) .. controls +(130:1) and +(90:1.7) ..  ($(o3)!.80!(i3)$) 
		($(o3)!.80!(i3)$) .. controls +(0,-1) and +(-1.7,0) ..  ($(o4)!.65!(i4)$) node[pos=0,left] {$\triangle$}
		($(o4)!.65!(i4)$) .. controls +(1.7,0) and +(0,-1) .. ($(o1)!.50!(i1)$)  
		($(o1)!.50!(i1)$) .. controls +(0,1) and +(-45:1em) ..  (o2.center) node[int,pos=.2] {$\beta$}; 
	\node[below] at (o4) {$(S,M,T)$};
}
{[yshift=-12cm,xshift=-6cm]
	\draw (0,0) circle [radius=5cm] ;
	\filldraw[fill=black!20] (0,0) circle [radius=1.5cm] ;
	
	\foreach \x [count=\n] in {0,1,2,3}{
		\node[name=i\n] at (90*\x:1.5cm) {};
		\node[name=o\n] at (90*\x:5cm) {};
	}
	\node[name=leftv,solid,label=125:$v'$] at (125:5cm) {};
	\node[solid] at (90:1.5cm) {};
	\node[name=rightv,solid,label=90:$v^{''}$] at (65:5cm) {};
	\node[solid] at (-90:5cm) {};
	
	\draw (leftv.center) .. controls +(-45:1) and +(90:2)  .. ($(o3)!.25!(i3)$) node[int] {$\gamma\dag$}
		 ($(o3)!.25!(i3)$) .. controls +(-90:2) and +(180:2) .. ($(o4)!.25!(i4)$)
		($(o4)!.25!(i4)$) .. controls +(0:2) and +(-90:2)  .. ($(o1)!.25!(i1)$)
		($(o1)!.25!(i1)$) .. controls +(90:2) and +(-45:2)  .. (rightv.center)  
		(rightv) -- (i2.center) node[int] {$\alpha\dag$} 
		(i2.center) .. controls +(130:1) and +(90:1.7) ..  ($(o3)!.80!(i3)$)  node[pos=.5,below left=.5em] {$\triangle\dag$}
		($(o3)!.80!(i3)$) .. controls +(0,-1) and +(-1.7,0) ..  ($(o4)!.65!(i4)$)
		($(o4)!.65!(i4)$) .. controls +(1.7,0) and +(0,-1) .. ($(o1)!.50!(i1)$)  
		($(o1)!.50!(i1)$) .. controls +(0,1) and +(-45:1em) ..  (rightv.center) node[int,pos=.2] {$\beta$}; 
	\node[below] at (o4) {$\cut v\gamma\alpha$};
}
{[yshift=-12cm,xshift=6cm]
	\draw (0,0) circle [radius=5cm] ;
	\filldraw[fill=black!20] (0,0) circle [radius=1.5cm] ;
	
	\foreach \x [count=\n] in {0,1,2,3}{
		\node[name=i\n] at (90*\x:1.5cm) {};
		\node[name=o\n] at (90*\x:5cm) {};
	}
	\node[name=rightv,solid,label=55:$v^{''}$] at (55:5cm) {};
	\node[solid] at (90:1.5cm) {};
	\node[name=leftv,solid,label=90:$v'$] at (115:5cm) {};
	\node[solid] at (-90:5cm) {};
	
	\draw (leftv.center) .. controls +(-135:2) and +(90:2)  .. ($(o3)!.25!(i3)$) 
		 ($(o3)!.25!(i3)$) .. controls +(-90:2) and +(180:2) .. ($(o4)!.25!(i4)$)
		($(o4)!.25!(i4)$) .. controls +(0:2) and +(-90:2)  .. ($(o1)!.25!(i1)$) node[int] {$\gamma\dag$}
		($(o1)!.25!(i1)$) .. controls +(90:2) and +(-90:2)  .. (rightv.center)  
		(leftv.center) -- (i2.center) node[int] {$\alpha$} 
		(i2.center) .. controls +(130:1) and +(90:1.7) ..  ($(o3)!.80!(i3)$) node[pos=.5,below left=.5em] {$\triangle\dag$}
		($(o3)!.80!(i3)$) .. controls +(0,-1) and +(-1.7,0) ..  ($(o4)!.65!(i4)$) 
		($(o4)!.65!(i4)$) .. controls +(1.7,0) and +(0,-1) .. ($(o1)!.50!(i1)$)  
		($(o1)!.50!(i1)$) .. controls +(0,1) and +(-45:1em) ..  (leftv.center) node[int,pos=.2] {$\beta\dag$}; 
	\node[below] at (o4) {$\cut v\beta\gamma$};
}
\end{tikzpicture}
\caption{All of the possible cuts at vertex $v$}\label{fig cut}
\end{figure}
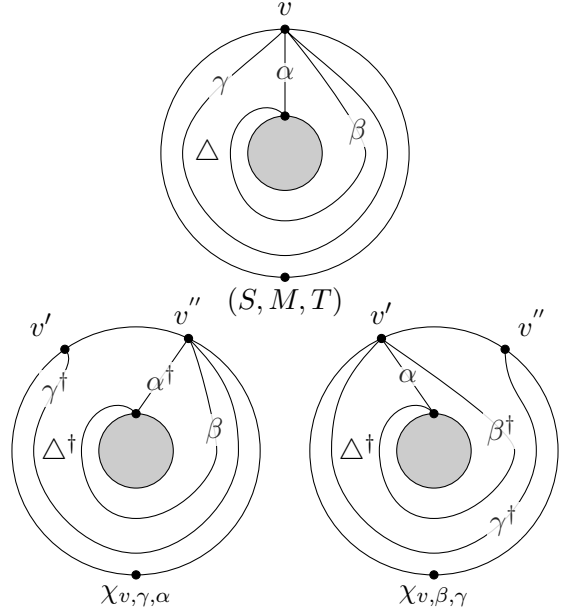

\subsection{Definition of surface algebras}
Let $(S,M^\dd,T^\dd)$ be a cut of $(S,M,T)$ given by the sequence 
$(\cut{v_i}{\alpha_i}{\beta_i})_{i=1,2\dots,t}$.  Note that each of the pairs $(\alpha_i,\beta_i)$ 
corresponds to a pair of vertices $(\alpha_i,\beta_i)$ in the quiver $Q_T$ and each triangle $\triangle_i$ to an arrow 
$\alpha_i\to \beta_i$ or $\beta_i\to \alpha_i$.

The collection $T^\dd$ is a partial triangulation of $(S,M^\dd)$; at each local cut
$\cut{v_i}{\alpha_i}{\beta_i}$ the arcs $\alpha_i^\dd$, $\beta_i^\dd$, $\gamma_i$ together
with the boundary segment between $v_i'$ and $v_i''$ form a quasi-triangle $\triangle_i^\dd$
in $T^\dd$. Choose a diagonal $\e_i$ for each of these quasi-triangles, as in
Figure~\ref{fig ei}, and let $\overline{T}^\dd= T^\dd \cup \{\e_1,\e_2,\dots, \e_t\}$. Then $\overline{T}^\dd
$ is a triangulation of $(S,M^\dd)$. Let $Q_{\overline{T}^\dd}$ be the associated quiver.
Note that each quasi-triangle $\triangle_i^\dd$ gives rise to a subquiver with four vertices
$\alpha_i^\dd$, $\beta_i^\dd$, $\gamma_i$, and $\e_i$, consisting of a 3-cycle with
vertices $\e_i$, $\gamma_i$ and either $\alpha_i^\dd$ or $\beta_i^\dd$, and one arrow
connecting the fourth vertex (either $\beta_i^\dd$ or $\alpha_i^\dd$) to the vertex $\e_i$. We may
suppose without loss of generality that these subquivers have the form
\[\begin{tikzpicture}
\node[name=a] at (0,0) {$\alpha_i^\dd$};
\node[name=e] at (1,0) {$\e_i$};
\node[name=g] at (3,0) {$\gamma_i$};
\node[name=b] at (2,-1) {$\beta_i^\dd$};
\path[->] (a) edge (e)  (e) edge (g) (g) edge (b) (b) edge (e);
\end{tikzpicture}\]
\begin{figure}
\centering
\begin{tikzpicture}[yscale=.5,xscale=.75]
\shade[shading angle=180] (0,4) rectangle (3,4.5);
\draw (0,0) -- (0,4) node[int] {$\alpha_i^\dd$} 
	-- (3,4)  --  (3,0) node[int] {$\beta_i^\dd$} 
	-- (0,0) node[int] {$\gamma_i$}
	-- (3,4) node[int] {$\e_i$};
\node[solid] at (0,0) {};
\node[solid] at (3,0) {};
\node[solid] at (0,4) {};
\node[solid] at (3,4) {};
\end{tikzpicture}
\caption{A choice for $\e_i$ in $\triangle\dag$}\label{fig ei}
\end{figure}

Now we define a new quiver $Q_{T^\dd}$ corresponding to the partial triangulation $T^\dd$
by deleting the vertices $\e_i$ and replacing each of the paths of length two given by
$\alpha_i^\dd \to \e_i\to \gamma_i$ by an arrow $\alpha_i^\dd \to \gamma_i$. Thus each quasi-triangle $\triangle_i^\dd$ in
${T^\dd}$ gives rise to a subquiver of
$Q_{T^\dd}$ of the form
$\alpha_i^\dd \to \gamma_i\to \beta_i^\dd$.

Next, we define relations on the quiver $Q_{T^\dd}$. First note that the quiver 
$Q_{\overline{T}^\dd}$ comes with the potential $W_{\overline{T}^\dd}$ and the Jacobian 
algebra $B_{\overline{T}^\dd}= kQ_{\overline{T}^\dd}/I_{\overline{T}^\dd}$, where $I_{\overline{T}^\dd}$
 is generated by the set $\calr$ consisting of all subpaths of length two of all oriented 
3-cycles. In particular, for each quasi-triangle $\triangle_i^\dd$, we have the three 
relations $\e_i\to \gamma_i\to \beta_i^\dd$, $\gamma_i\to \beta_i^\dd\to \e_i$,
and $\beta_i^\dd\to \e_i\to \delta_i$. Denote by $\calr_i$ the set of these three
relations. Define $I_{T^\dd}$ to be the two-sided ideal generated by all relations in
	\[\left(\calr\setminus \Big( \bigcup_{i=1}^t \calr_i\Big) \right) 
	        \cup \left( \bigcup_{i=1}^t \{\alpha_i^\dd\to \gamma_i\to \beta_i^\dd\}\right).\]
Thus each $\triangle_i^\dd$ corresponds to a subquiver 
$\alpha_i^\dd\to \gamma_i\to \beta_i^\dd$ with a zero relation. 
\begin{dfn}\label{def surface algebra}
A \df{surface algebra of type}  $(S,M)$ is a bound quiver algebra $B_{T^\dd}=kQ_{T^\dd}/I_{T^\dd}$ where
$(S,M^\dd,T^\dd)$ is a cut of a triangulated unpunctured surface $(S,M,T)$.
\end{dfn}

\subsection{Properties of surface algebras}
\begin{lemma}\label{lem admissible cut}
If $(S,M^\dd,T^\dd)$ is obtained from $(S,M,T)$ by the sequence 
of local cuts $(\cut{v_i}{\alpha_i}{\beta_i})_{i=1,\dots,t}$, then $Q_{T^\dd}$ is isomorphic to the quiver 
obtained from $Q_T$ by deleting the arrows $\beta_i\to \alpha_i$ for $i=1,2,\dots,t$.
\end{lemma}
\begin{proof}
This follows immediately from the construction.
\end{proof}

\begin{thm}\label{thm adm cut}
If $(S,M^\dd,T^\dd)$ is an admissible cut of $(S,M,T)$ then
\begin{enuma}
\item  $Q_{T^\dd}$ is an admissible cut of $Q_T$;
\item  $B_{T^\dd}$ is of global dimension at most two,  and $B_{T^\dd}$ is of global dimension  
	 one if and only if $B_{T^\dd}$is a hereditary algebra of type $\AA$ or $\tilde \AA$;
\item The tensor algebra of $B_{T^\dd}$ with respect to the $B_{T^\dd}$-bimodule 
		\[\Ext^2_{B_{T^\dd}}(DB_{T^\dd},B_{T^\dd})\]
	 is isomorphic to the Jacobian algebra $B_T$. 
\end{enuma}
\end{thm}

\begin{proof}
Part (a).  The oriented 3-cycles in $Q_T$ are precisely the chordless cycles in $Q_T$ and,
by Lemma~\ref{lem admissible cut}, $Q_{T\dag}$ is obtained from $Q_T$ by
deleting exactly one arrow in each chordless cycles; this shows (a).

Part (b). Since $Q_{T^\dd}$ does not contain any oriented cycles, the ideal $I_{T^\dd}$ is generated by
monomial relations which do not overlap. This immediately implies $\gdim B_{T^\dd}\leq 2$.

If $B_{T^\dd}$ is of global dimension at most one, then the ideal $I_{T^\dd}$ is trivial and
hence the ideal $I_T$ is trivial too. It follows that $Q_T$ has no oriented cycles, $T^\dd=T$,
$B_{T^\dd}=B_T$, and that $T$ is a triangulation  without internal triangles. The only unpunctured surfaces
that admit such a triangulation are the disc and the annulus, corresponding to the case where
$B_{T^\dd}$ is hereditary algebra of type $\AA$ or $\tilde \AA$, respectively.

Part (c). Let $\widetilde{B_{T^\dd}}$ denote the tensor algebra. It follows from \cite{ABS} that its
quiver is obtained from $Q_{T^\dd}$ by adding on arrow $\beta_i^\dd\to \alpha_i^\dd$ for each
relation $\alpha_i^\dd \to \gamma_i\to \beta_i^\dd$; thus the quiver of $\widetilde{B_{T^\dd}}$ is
isomorphic to the quiver $Q_T$. Moreover, it follows from \cite[Theorem 6.12]{Keller-deformed}  
that $\widetilde{B_{T^\dd}}$ is a Jacobian algebra with potential $\tilde W$ given by the
sum of all 3-cycles $\alpha_i^\dd\to \gamma_i\to \beta_i^\dd \to \alpha_i^\dd$; thus
$\widetilde{B_{T^\dd}}\cong B_T$.
\end{proof}

\begin{cor}\label{cor 3.6}
The admissible cut surface algebras of the disc with $n+3$ marked points are precisely the 
iterated tilted algebras of type $\AA_n$ whose global dimension is at most two.
\end{cor}

\begin{proof}
In \cite{BFPPT}, the authors have shown that the quotients by an admissible cut of a cluster-tilted
algebra of type $\AA_n$ are precisely the iterated tilted algebras of type $\AA_n$  of global 
dimension at most two.  By \cite{CCS}, the cluster-tilted algebras of type $\AA_n$ are precisely the 
algebras $B_T$, where $T$ is a triangulation of the disc with $n+3$ marked points.  The result now 
follows from Theorem~\ref{thm adm cut}.
\end{proof}

\begin{thm}
Every surface algebra is isomorphic to the endomorphism algebra of a partial cluster tilting 
object in a generalized cluster category.  More precisely, if the surface algebra $B_{T^\dd} $ is given by the 
cut $(S,M^\dd,T^\dd)$ of the triangulated surface $(S,M,T)$, then  \[B_{T^\dd}\cong \End_{\calc{(S,M^\dd)}}T^\dd,\]
where $T^\dd$ denotes the object in the generalized cluster category $\calc{(S,M^\dd)}$ corresponding to $T^\dd$. 
\end{thm}

\begin{proof}
Let $\overline{T}^\dd$ be the completion of $T^\dd$ as in the construction of the quiver $Q_{T^\dd}$. 
By \cite{FST}, the triangulation $\overline{T}^\dd$ corresponds to a cluster in the cluster algebra of $(S,M^\dd)$, hence $\overline{T}^\dd$ also corresponds to a cluster-tilting object in the generalized cluster category $\mathcal{C}{(S,M^\dd)}$, see \cite{Amiot, BZ}.
 Thus  $T^\dd$ is a partial 
cluster-tilting object in $\calc{(S,M^\dd)}$. The endomorphism algebra $\End_{\calc{(S,M^\dd)}}\overline{T}^\dd$ of 
the cluster-tilting object $\overline{T}^\dd$ is isomorphic to the Jacobian algebra $B_{\overline{T}^\dd}$.

Let $\cut{v_i}{\alpha_i}{\beta_i}$ be one of the local cuts that defines the cut $(S,M^\dd,T^\dd)$.  
In the quiver $Q_{\overline{T}^\dd}$, we have the corresponding subquiver
\[\begin{tikzpicture}
\node[name=a] at (0,0) {$\alpha_i^\dd$};
\node[name=e] at (1,0) {$\e_i$};
\node[name=g] at (3,0) {$\gamma_i$};
\node[name=b] at (2,-1) {$\beta_i^\dd$};
\path[->] (a) edge (e)  (e) edge (g) (g) edge (b) (b) edge (e);
\end{tikzpicture}\]
and there are no other arrows in $Q_{\overline{T}^\dd}$ starting or ending at $\e_i$. Each vertex in this
subquiver corresponds to an indecomposable summand of the cluster-tilting object $\overline{T}^\dd$, and
each non-zero path corresponds to a non-zero morphism between the indecomposable summands
associated to the endpoints of the path.  Thus in $\calc{(S,M^\dd)}$, there are non-zero morphisms
$f\colon \beta_i^\dd\to \gamma_i$, $g\colon \gamma_i\to \e_i$, $h_i\colon \e_i \to \beta_i$,
$u_i\colon \e_i \to \alpha_i^\dd$, and the compositions $g_if_i$, $h_ig_i$, $f_ih_i$ are zero, but the 
composition $u_ig_i$ is non-zero. 

Removing the summand $\e_i$ from the cluster-tilting object $\overline{T}^\dd$ and considering 
$\End_{\calc{(S,M^\dd)}}(\overline{T}^\dd\setminus \e_i)$, the only non-zero morphisms between the summands
$\alpha_i^\dd,$ $\beta_i^\dd$, and $\gamma_i$ are $f_i$ and $u_ig_i$, and the composition
$(u_ig_i)f_i$ is zero, since $g_if_i$ is zero.  Thus, locally, the quiver of $\End_{\calc{(S,M^\dd)}}(\overline{T}^\dd\setminus \e_i)$ is
$\alpha_i^\dd \to \gamma_i\to \beta_i^\dd$ and the composition of the two arrows is zero in 
$\End_{\calc{(S,M^\dd)}}(\overline{T}^\dd\setminus \e_i)$.

The endomorphism algebra $\End_{\calc{(S,M^\dd)}}(T^\dd)$ of the partial cluster tilting object $T^\dd$ is obtained
by a finite number of such local transformations, each of which is corresponding to a local cut
$\cut{v_i}{\alpha_i}{\beta_i}$. Thus $\End_{\calc{(S,M)}}(T^\dd) = kQ_{T^\dd}/I_{T^\dd} = B_{T^\dd}$. 
\end{proof}

\begin{thm}
Every surface algebra is gentle.
\end{thm}

\begin{proof}
Let $B_{T^\dd}$ be the surface algebra obtained by a cut $(S,M^\dd,T^\dd)$ of a triangulated surface 
$(S,M,T)$.  The algebra $B_T$ is gentle by  Proposition \ref{prop 2.7}, and the result then follows from Lemma \ref {lem admissible cut} and Proposition \ref{prop 2.8}.
\end{proof}

\section{Computing the AG-invariant of surface algebras}\label{sect 4}
In this section, we give an explicit formula for the AG-invariant of an arbitrary surface
algebra in terms of the surface. The key idea is to interpret the permitted threads as
complete fans and the forbidden threads as triangles or quasi-triangles in the partial
triangulation $T\dag$.

\subsection{Permitted threads and complete fans}
Let $(S,M\dag,T\dag)$ be a cut of a triangulated unpunctured surface $(S,M,T)$, and let $v$ be a 
point in $M\dag$.  Let $v',v''$ be two points on the boundary, but not in $M\dag$, such that
there is a curve $\delta$ on the boundary from $v'$ to $v''$ passing through $v$, but not
passing through any other point of $M\dag$. Let $\gamma$ be a curve from $v'$ to $v''$ which is
homotopic to $\delta$ but disjoint from the boundary except for its endpoints, such that
the number of crossings between $\gamma$ and $T\dag$ is minimal. The sequence of arcs $\tau_1$,
$\tau_2$, \dots $\tau_j$ in $T\dag$ that $\gamma$ crosses in order is called the \df{complete
fan at $v$ in $T\dag$.} The arcs $\tau_2,\dots,\tau_{j-1}$ are called the \df{interior arcs of
the fan} and $j$ is called the \df{width of the fan}. See Figure~\ref{fig fan def}

Note that the sequence is of width zero if no arc in $T\dag$ is incident to $v$.
Fans of width zero are called \df{empty fans}, and fans of width one are called 
\df{trivial fans}. Every arc in $T^\dd$ is contained in exactly two non-empty 
complete fans, namely the two fans at the two endpoints of the arc.
\begin{figure}
\begin{tikzpicture}
\shade[shading=axis,shading angle=180]  (0,0) rectangle (4,.25);
\draw (0,0) -- (4,0);
\node[coordinate,name=m1,label={[fill=white,shape=circle,inner sep=.5,outer sep=1.5]90:$v'$}] at (1,0) {};
\node[solid,name=m2,label={[fill=white,shape=circle,inner sep=.5,outer sep=1.5]90:$v$}] at (2,0) {};
\node[coordinate,name=m3,label={[fill=white,shape=circle,inner sep=.5,outer sep=1.5]90:$v''$}] at (3,0) {};
\draw (m1) edge[bend right] (m3);
\clip (0,.1) rectangle (4,-1.5); 
\foreach \a in {35,55,...,155}
	\draw (m2.center) -- +(-\a:3cm);
\node[fill=white,opacity=.8,inner sep=.8] at ($(m2)+(-35:1.5cm)$) {$\tau_j$};
\node[fill=white,opacity=.8,inner sep=.8] at ($(m2)+(-135:1.5cm)$) {$\tau_2$};
\node[fill=white,opacity=.8,inner sep=.8] at ($(m2)+(-155:1.5cm)$) {$\tau_1$};
\end{tikzpicture}
\caption{The fan at $v$}\label{fig fan def}
\end{figure}
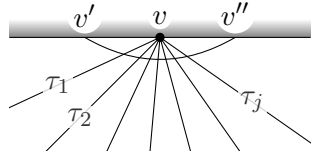

\begin{lemma}\label{lem permitted to v}
The construction of $Q_{T^\dd}$ induces a bijection between the set $\calh$ of permitted threads in 
$Q_{T^\dd}$ and the non-empty complete fans  of $(S,M^\dd,T^\dd)$.  Moreover, under this bijection,
the trivial permitted threads correspond to the trivial fans.
\end{lemma}
\begin{proof}
From the construction of $Q_{T^\dd}$ it follows that the non-trivial permitted threads of $Q_{T^\dd}$
correspond to the complete fans in $(S,M^\dd,T^\dd)$ of width at least 2.  

Now consider a trivial permitted thread; this corresponds to a vertex $x\in Q_{T^\dd}$ with at most
one arrow starting at $x$ and at most one arrow ending at $x$.  By the construction of $Q_{T^\dd}$, the vertex
$x$ corresponds to an arc in $T^\dd$, and this arc is contained in exactly two complete fans, one at each 
endpoint of the arc.  If both fans are non-trivial, then this configuration generates either two arrows
ending at the vertex $x$ in $Q_{T^\dd}$ or two arrows starting at the vertex $x$, which contradicts our 
assumption that $x$ corresponds to a trivial permitted thread.  Hence one of the two fans must be trivial.

Since we are excluding the case where $(S,M)$ is a disc with four marked points, it follows from the
 construction of $T^\dd$, that it is impossible that both fans at the endpoint of an arc are trivial.  
 This shows that the trivial permitted threads are in bijection with the trivial complete fans.
\end{proof}

\subsection{Forbidden threads, triangles and quasi-triangles}
\begin{lemma}\label{lem forbidden to tri}
The construction of $Q_{T^\dd}$ induces a bijection between the set of forbidden threads  of length 
at most two of $Q_{T^\dd}$ and the set of non-internal triangles and \quasitri{s} in  $T^\dd$.
Moreover, under this bijection
\begin{enuma}
\item forbidden threads of length two correspond to  \quasitri{s},
\item forbidden threads of length one correspond to triangles with 
	exactly one side on the boundary,
\item forbidden threads of length zero correspond to triangles with
	exactly two sides on the boundary.
\end{enuma}
\end{lemma}

\begin{remark}
\begin{enuma}
\item Each internal triangle in $T\dag$ gives rise to three forbidden threads of length three.
\item There are no forbidden threads in $Q_{T\dag}$ of length greater than three.
\item If $T\dag=T$ then there are no forbidden threads of length two.
\end{enuma}
\end{remark}

\begin{proof}
Forbidden threads of length two can only occur at the local cuts, since all relations in
$B_T$ occur in oriented 3-cycles, and hence give forbidden threads of
length three. Thus the forbidden threads of length two are precisely the
paths $\alpha_i^\dd \to \gamma_i\to \beta_i^\dd$ in $Q_{T^\dd}$ which were obtained from the 
3-cycles 
 $\alpha_i\to \gamma_i\to \beta_i \to \alpha_i$
by the cuts $\cut{v_i}{\alpha_i}{\beta_i}$.   These paths correspond to the \quasitri\ $\triangle_i^\dd$. This shows (a).

A forbidden thread of length one in $Q_{T^\dd}$ is an arrow $x\to y$ which does not appear in any
relation, which means that the arcs corresponding to $x$ and $y$ in the triangulation $T^\dd$ bound
a triangle which is not an internal triangle and also not a \quasitri\ and which  has two interior sides
(corresponding to $x$ and $y$). This shows (b).

Finally, a trivial forbidden thread in $Q_{T^\dd}$ is a vertex $x$ that is either a sink  with only one arrow 
ending at $x$, or a source with only one arrow starting at $x$, or the middle vertex of a zero relation of 
length two and the two arrows from this zero relation are the only two arrows at $x$. 
If $x$ is such a source or sink, then  it follows  from the construction of $Q_{T^\dd}$ that the triangle on one side of the 
arc corresponding to $x$ in $T^\dd$ must not generate an arrow in $Q_{T^\dd}$. Thus $x$ is a side
of a triangle that has two sides on the boundary.  On the other hand, if $x$ is such a middle vertex of a zero relation, then 
the triangle on one of side of the arc corresponding to $x$ in $T^\dd$ must be an internal triangle or   a \quasitri.  
On the other side, we must have a triangle with two sides  on the boundary, since there are no further 
arrows at $x$. 

Conversely, given any triangle with two sides on the boundary and the third side $\gamma$ corresponding to
a vertex $x\in Q_{T^\dd}$, we can deduce that $x$ is a trivial forbidden thread, because, on the other side 
of $\gamma$, we have either a triangle with one side on the boundary or an internal triangle, 
respectively \quasitri, thus $x$ is either a vertex with only one arrow incident to it or the middle vertex
of a zero relation with no further arrows incident to it.
\end{proof}

If $F\in \calf$ is a forbidden thread, we denote by $\triangle_F$ the corresponding triangle or quasi-triangle in $Q_{T^\dd}$, 
and if $H\in \calh$ is a permitted thread, we denote by $v_H$ the marked point in $H$ such that $H$ 
corresponds to the complete fan at $v_H$. 

\subsection{AG-invariant}

We now want to compute the AG-invariant of $B_{T^\dd}$ in terms of the partially triangulated surface
$(S,M^\dd,T^\dd)$.  To do so, we need
to describe how to go from a forbidden thread to a the following  permitted thread and from a 
permitted thread to the following forbidden thread as in the AG-algorithm.  

\begin{lemma}\label{lem varphi(F)}
Let $F\in \calf$ be a forbidden thread in $Q_{T^\dd}$, and let $x=s(F)$ be its starting vertex in
$Q_{T^\dd}$. Then there exists a unique permitted thread $\varphi(F)\in \calh$ such that $s(\varphi(F))=x$
and $\sigma(\varphi(F))=-\sigma(F)$. Moreover, the marked point $v_{\varphi(F)}$ of the complete
fan of $\varphi(F)$ is the unique vertex of $\triangle_F$ that is the starting point of the arc
corresponding to $x$ in $T^\dd$, with respect to the counterclockwise orientation of $\triangle_F$, see
Figure~\ref{fig varphi(F)}.
\end{lemma}
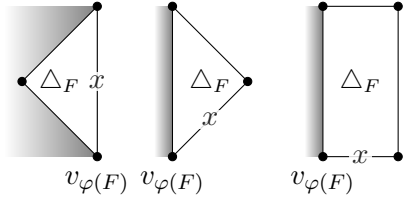
\begin{figure}
\begin{tikzpicture}
{[]
	\shade[shading=axis,shading angle=-90] (1,-1)	rectangle (-.25,1);
	\filldraw[fill=white] (0,0) node[pos=0,solid] {} 
		-- +(1,1)  node[pos=1,solid] {} 
	          -- +(1,-1) node[pos=1,solid,label=below:$v_{\varphi(F)}$]{} node[int] {$x$} 
	          -- (0,0);
	\node[font=\small] at (.5,0) {$\triangle_F$};
}
{[xshift=2cm]
	\shade[shading=axis,shading angle=-90] (0,1)	rectangle (-.25,-1);
	\filldraw[fill=white] (0,1)  
		-- ++(1,-1)  node[pos=0,solid] {} node[pos=1,solid] {} 
	          -- ++(-1,-1) node[pos=1,solid,label=below:$v_{\varphi(F)}$]{} node[int] {$x$}
	          -- cycle;
	\node[font=\small] at (.5,0) {$\triangle_F$};
}
{[xshift=4cm]
	\shade[shading=axis,shading angle=-90] (0,1)	rectangle (-.25,-1);
	\filldraw[fill=white] (0,1)  
		-- ++(1,0)  node[pos=0,solid] {} node[pos=1,solid] {} 
		-- ++(0,-2) node[pos=1,solid] {} 
	          -- ++(-1,0) node[pos=1,solid,label=below:$v_{\varphi(F)}$]{} node[int] {$x$}
	          -- cycle;
	\node[font=\small] at (.5,0) {$\triangle_F$};
}
\end{tikzpicture}
\caption{ The relative position of the marked point $v_{\varphi(F)}$; on the left, $F$ is the
trivial forbidden thread; in the middle, $F$ is of length one; and, one the right, $F$ is of length two.}
\label{fig varphi(F)} 
\end{figure}

\begin{proof} The existence and uniqueness of $\varphi(F)$ follows already from \cite{AG}, but we
include a proof here, for convenience. If $F$ is a trivial forbidden thread, then there is at most
one arrow start at $x$, hence there is a unique permitted thread $\varphi(F)$ starting at $x$. By
Lemma~\ref{lem permitted to v}, $\varphi(F)$ corresponds to a complete fan in which the first arc
corresponds to $x$. It follows that the vertex of this fan is the one described in the Lemma.

Now suppose that $F$ is not trivial, and let $\alpha$ be its initial arrow. Then there are two
permitted threads starting at $x$, one of which has also $\alpha$ as initial arrow. However, since
$\sigma(F)=\sigma(\alpha)$, the condition $\sigma(\varphi(F))=-\sigma(F)$ excludes the possibility
that $\varphi(F)$ starts with $\alpha$; thus $\varphi(F)$ is the other permitted thread starting at
$x$.  Again using Lemma~\ref{lem permitted to v}, we see that $\varphi(F)$ corresponds to the 
complete fan whose vertex is the marked point $v_{\varphi(F)}$ described in the Lemma. Note that
if $\alpha$ is the only arrow starting at $x$, then $\varphi(F)$ trivial. 
\end{proof}

\begin{lemma}\label{lem psi(H)}
Let $H\in \calh$ be a permitted thread in $Q_{T^\dd}$ and let $y=t(H)$ be its terminal vertex in
$Q_{T^\dd}$. Then there exists a unique forbidden thread $\psi(H)\in \calf$ such that $t(\psi(H))=y$
and $\e(\psi(H))=-\e(H)$. Moreover, the triangle or \quasitri\ $\triangle_{\psi(H)}$ of $\psi(H)$ is
the unique triangle or \quasitri\ which is adjacent to the arc corresponding to $y$ and incident to
$v_H$, see Figure~\ref{fig psi(H)}.
\end{lemma}
\begin{figure}
\begin{tikzpicture}
{[]
\draw (0,0) node[solid,label=135:$v_H$] {}
	-- ++(-30:2.25cm) 
		node[int] {$y$}
		node[solid,pos=1] {} 
		node[left=.33cm,pos=1] {$\triangle_{\psi(H)}$}
	-- ++(-150:2.25cm) 
		node[solid,pos=1,name=b] {} 
	-- cycle;
\draw (0,0)  -- (20:2cm) 
	(0,0)  -- (10:2cm)
	(0,0)  -- (0:2cm);
\draw[loosely dotted] (-1:1.9cm).. controls (-15:2cm)  .. (-29:1.9cm)
	node[pos=.5,right=.25cm] {$\Fan(v_H)$};
	
{[on background layer] \shade[shading=axis,shading angle=-90] (-.25,0) rectangle (b); }
}	
{[xshift=4.5cm]
\draw (0,0) node[solid,label=135:$v_H$] {}
	-- ++(-25:2.25cm) 
		node[int] {$y$}
		node[solid,pos=1] {} 
	--++(-90:.75cm)
		node[solid,pos=1,name=c] {} 
		node[above left=.33cm,pos=1] {$\triangle_{\psi(H)}$}
	(0,0) -- ++(-90:2.5cm) 
		node[solid,pos=1,name=b] {} 
	-- (c);
\draw (0,0)  -- (20:2cm) 
	(0,0)  -- (10:2cm)
	(0,0)  -- (0:2cm);
\draw[loosely dotted] (-1:1.9cm).. controls (-15:2cm)  .. (-24:1.9cm)
	node[pos=.5,right=.25cm] {$\Fan(v_H)$};
	
{[on background layer] \shade[shading=axis,shading angle=-90] (-.25,0) rectangle (b); }
}	
\end{tikzpicture}
\caption{}\label{fig psi(H)}
\end{figure}

\begin{proof}
Again, the existence and uniqueness already follows from \cite{AG}, but we include a
proof here too. 
It follows from Lemma~\ref{lem permitted to v} that the side of the triangle $\triangle_{\psi(H)}$
which is following $y$ in the counterclockwise order must be a boundary segment, since the fan at $v_H$
is a complete fan and $y=t(H)$. It follows from Lemma~\ref{lem forbidden to tri} that the forbidden
thread corresponding to $\triangle_{\psi(H)}$ ends in $y$. Let $b$ be the terminal arrow in $H$.
If $\triangle_{\psi(H)}$ has two sides on the boundary, then $\psi(H)$ is a trivial forbidden thread
and $\e(\psi(H))=-\e(b)=-\e(H)$. Note that in this case, $b$ is the only arrow in $Q_{T^\dd}$
ending  at $y$, and, since $\e(b)=\e(H)$, it follows that $\psi(H)$ is unique.  

Suppose now that $\triangle_{\psi(H)}$ has only one side on the boundary.  Since this boundary
segment follows $y$ in the counterclockwise orientation, $\triangle_{\psi(H)}$ induces an arrow
$b'$ in $Q_{T^\dd}$, with $t(b')=y$, and this arrow is the terminal arrow of $\psi(H)$.  
Since $b$ and $b'$ end at the same vertex, we have $\e(\psi(H)) = \e(b') = -\e(b) = -\e(H)$.
This shows the existence of $\psi(H)$, and its uniqueness follows as above. 
\end{proof}

We are now able to prove the main result of this section, the computation of the AG-invariant for
an arbitrary surface algebra. 

For any triangulated surface $(S,M,T)$ and any boundary component $C$ of $S$, denote by $M_{C,T}$
the set of marked points on $C$ that are incident to at least one arc in $T$. Let $n(C,T)$ be the
cardinality of $M_{C,T}$, and let $m(C,T)$ be the number of boundary segments on $C$ that 
have both endpoints in $M_{C,T}$.

\begin{thm}\label{thm AG calc}
Let $B=B_{T^\dd}$ be a surface algebra of type $(S,M,T)$ given by a cut $(S,M^\dd,T^\dd)$.  Then
the AG-invariant $AG(B)$ can be computed as follows:
\begin{enuma}
\item The ordered pairs $(0,3)$ in $AG(B)$ are in bijection with the internal triangles in $T^\dd$, and 
there are no ordered pairs $(0,m)$ with $m\neq 3$.
\item The ordered pairs $(n,m)$ in $AG(B)$ with $n\neq 0$ are in bijection with the boundary 
components of $S$. Moreover, if $C$ is  a boundary component, then the corresponding ordered
pair $(n,m)$ is given by
\[n= n(C,T) +\ell , \quad 	m= m(C,T) +2\ell, \]
where $\ell$ is the number of local cuts $\chi_{v,\alpha,\beta}$  in $(S,M^\dd,T^\dd)$ such that $v$ 
is a point on $C$. 
\end{enuma}
\end{thm}

\begin{proof}
Part (a) follows directly from the construction of $B_{T^\dd}$.  To show (b), let 
	\[(H_0,F_0,H_1,F_1,\dots,H_{n-1},F_{n-1},H_n=H_0)\]
be a sequence obtained by the AG-algorithm.  Thus each $H_i$ is a permitted thread in 
$Q_{T^\dd}$, each $F_i$ is a forbidden thread, $F_i=\psi(H_i)$, and $H_{i+1} = \varphi(F_i)$, 
where $\varphi$ and $\psi$ are the maps described in the Lemmas~\ref{lem varphi(F)} 
and~\ref{lem psi(H)} respectively. 

By Lemma~\ref{lem permitted to v}, each permitted thread $H_i$ corresponds to a 
non-empty complete fan, thus to a marked point $v_{H_i}$ in $M^\dd$.  On the other
hand, Lemma~\ref{lem forbidden to tri} implies that each forbidden thread $F_i$ 
corresponds to a triangle or \quasitri\ $\triangle_{F_i}$ in $T^\dd$. 

Let $C$ be the boundary component containing $v_{H_0}$.  Lemma~\ref{lem psi(H)} implies
that $\triangle_{F_0}$ contains a boundary segment incident to $v_{H_0}$, and it follows then from 
Lemma~\ref{lem varphi(F)} that $v_{H_1}\in M^\dd$ is a marked point on the same boundary 
component $C$.  Note that if $F_0$ is a non-trivial forbidden thread then $v_{H_0}$ and $v_{H_1}$
are the two endpoints of the unique boundary  segment in $\triangle_{F_0}$, and, if $F_0$ is a 
trivial forbidden thread, the $\triangle_{F_0}$ is a triangle with two sides on the boundary and 
$v_{H_0}$, $v_{H_1}$ are the two endpoints of the side of $\triangle_{F_0}$ that is not on the boundary. 

Recursively, we see that each of the marked points $v_{H_i}\in M^\dd$ lies on the boundary component
$C$ and that the set of points $\{v_{H_0},v_{H_1},\dots,v_{H_{n-1}}\}$ is precisely the set of marked 
points in $M^\dd$ that lie on $C$ and that are incident to at least one arc in $T^\dd$. In particular,
$n=n(C,T)+\ell$.  

Recall that the number $m$ in the ordered pair $(n,m)$ is equal to the sum of the number of 
arrows appearing in the forbidden threads $F_0,F_1,\dots,F_{n-1}$.  For each $i$, the number of 
arrows in $F_i$ is zero if $\triangle_{F_i}$ has two sides on the boundary component $C$; it is one if
$\triangle_{F_i}$ is a triangle with one side on $C$; and it is two if $\triangle_{F_i}$ is a \quasitri\
with one side on $C$.  Taking the sum, we see that $m$ is the number of triangles in $T^\dd$
that have exactly one side on $C$ plus twice the number of \quasitri\ in $T^\dd$ that have
exactly one side on $C$.  Thus $m$ is equal to the number of boundary segments in $(S,M,T)$
on $C$ that have both endpoints in $M_{C,T}$ plus twice the number of local cuts $\cut{v}{\za}{\zb}$ with $v$ a marked point in $C$.
\end{proof}

\begin{remark}
The theorem holds for arbitrary cuts of $(S,M,T)$, thus, in particular, it computes the AG-invariant of the 
Jacobian algebra $B_T$ corresponding to the uncut triangulated surface $(S,M,T)$. 
\end{remark}

\begin{cor}
Any admissible cut surface algebra of the disc with $n+3$ marked points is derived
equivalent to the path algebra $A$ of the quiver
\[ 1\xrightarrow{\ \, a_1 \ } 2 \xrightarrow{\ \, a_2\ } 3 \xrightarrow{\,\ a_3\ } \cdots \xrightarrow{a_{n-2}} n-1 \xrightarrow{a_{n-1}} n\]
\end{cor}

\begin{proof}
This statement already follows from Corollary \ref{cor 3.6}, since tilting induces an equivalence of derived categories. We give here an alternative proof using Theorem \ref{thm AG calc}. The  algebra $A$ has AG-invariant $(n+1,n-1)$, since in the AG algorithm we get the sequence
\[\begin{array}{lll} H_0=a_1a_2\cdots a_{n-1} & \quad\quad& F_0=e_n\\ H_1=e_n && F_1=a_{n-1}\\ H_2=e_{n-1}&& F_2=a_{n-2}\\ \qquad\vdots && \qquad\vdots \\H_n=e_1 && F_{n}=e_1\\ H_{n+1}=H_0\end{array}\]
which contains $n+1$ permissible threads and $n-1$ arrows in the forbidden threads.

On the other hand, Theorem \ref{thm AG calc} implies that the AG-invariant of an admissible cut surface algebra $B_{T^\dd}$ of the disc is given by \[(n(C,T)+\ell\ , \ m(C,T)+2\ell).\] It follows from an easy induction that the number of internal triangles in $T$ is equal to $n+1-n(C,T).$ Since the cut is admissible, this is also the number of local cuts. Thus 
\[ n(C,T)+\ell=n+1.\]
Moreover, $m(C,T)$ is equal to the number of boundary segments minus $2(n+3-n(C,T))$, thus $m(C,T)=2n(C,T)-n-3$, and it follows that \[m(C,T)+2\ell =n-1.\]
Alternatively, since there is only one ordered pair $(n,m)$ in the AG-invariant, the number $m$ is the number of arrows in the quiver, which is equal to $n-1$, since the quiver of an admissible cut from the disc with $n+3$ marked points is a tree on $n$  vertices.

Thus the AG-invariant of $B_{T^\dd}$ is equal to the AG-invariant of $A$. Since the quivers of both algebras have no cycles, now the result follows from Theorem \ref{thm AG}.
\end{proof}

\begin{remark}
It also follows from Theorem \ref{thm AG calc} that for surfaces other than the disc, the surface algebras obtained by admissible cuts of a fixed triangulation are in general not all derived equivalent.
\end{remark}

\subsection{Surface algebras of annulus type}
We make some observations for surface algebras coming from admissible cuts of a triangulated annulus.

\begin{cor}\label{cor annulus AG classes}
Let $S$ be an annulus, and let $B_1$, $B_2$ be surface algebras of type $(S,M,T)$ obtained by 
two admissible cuts of the same triangulation.  Then 
\begin{enuma}
\item $B_1$ and $B_2$ are derived equivalent 
if and only if $B_1 $ and $B_2$ have the same AG-invariant.
\item If on each boundary component the number of local cuts is the same in the
two admissible cuts then $B_1$ and $B_2$ are derived equivalent.
\end{enuma}
\end{cor}

\begin{proof}
(a) The quivers of $B_1$ and $B_2$ have exactly one cycle, since they are admissible cuts of the annulus.
Thus Theorem \ref{thm AG} implies  that $\AG(B_1)=\AG(B_2)$ if and only if $B_1$ and $B_2$ are derived
equivalent. 

(b)  Theorem \ref{thm AG calc} shows that $\AG(B_1)=\AG(B_2)$ if  the number
of local cuts on any given boundary component is the same in both admissible cuts. 
\end{proof}

\begin{remark}{In \cite{AO}, the authors associate a weight to each of the algebras appearing in Corollary \ref{cor annulus AG classes} and show that two such algebras are derived equivalent if and only if they have the same weight (respectively the same absolute weight if the marked points are equally distributed over the two boundary components).}
\end{remark}

\begin{ex}
We end this section with an example coming from an annulus.  Let $(S,M,T)$ be the triangulated surface
given in Figure~\ref{fig ex annulus} and $B$ the corresponding cluster-tilted algebra of type $\tilde\AA_4$ with the 
following quiver
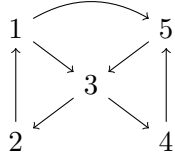
\[\begin{tikzpicture}[rotate=-90]
\node[name=1] at (1,0) {1};
\node[name=2] at (2.5,0) {2};
\node[name=3] at (1.75,1)  {3};
\node[name=4] at (2.5,2) {4};
\node[name=5] at (1,2) {5};
\path[->]  (1) edge[bend left] (5) edge (3) 
	(5) edge (3)
	(3) edge (2) edge (4)
	(4) edge (5)
	(2) edge (1);
\end{tikzpicture}\]
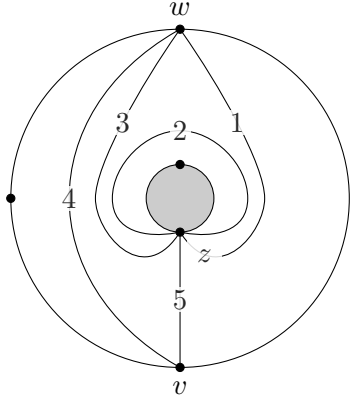
\begin{figure}
\begin{tikzpicture}[scale=.45]
\draw (0,0) circle[radius=5cm];
\filldraw[fill=black!20] (0,0) circle[radius=1cm];
\node[solid,name=a,label=above:$w$] at (90:5cm) {} ;
\node[solid,name=b] at (180:5cm) {} ;
\node[solid,name=c,label=below:$v$] at (270:5cm) {} ;
\node[solid,name=1] at (90:1cm) {} ;
\node[solid,name=2] at (270:1cm) {} ;
\draw (a) .. controls +(210:5cm) and +(150:5cm) .. (c) node[int] {4}
	 (a) ..controls +(235:2cm) and +(90:1cm) .. (180:2.5cm) node[int] {3}
	 	(180:2.5cm) ..controls +(270:1cm) and +(235:2cm) .. (2)
	 (a) ..controls +(-55:2cm) and +(90:1cm) .. (0:2.5cm) node[int] {1}
	 	(0:2.5cm) ..controls +(270:1cm) and +(-55:2cm) .. (2)
	(2) .. controls +(-10:1cm) and +(270:1cm) .. (0:2cm) 
		(0:2cm) .. controls +(90:1cm) and +(0:1cm) .. (90:2cm) node[int,pos=1] {2}
		(90:2cm) .. controls +(180:1cm) and +(90:1cm) .. (180:2cm)
		(180:2cm) .. controls +(270:1cm) and +(190:1cm) .. (2)
	(2) -- (c) node[int] {5};
\node[outer sep=.1cm,fill=white,opacity=.85,below right] at (270:1cm){$z$};
\end{tikzpicture}
\caption{Triangulated annulus}\label{fig ex annulus}
\end{figure}
Applying Theorem~\ref{thm AG calc}, we see that  $\AG(B)$ is given by the ordered pairs 
$\{ (0,3) , (0,3), (1,0), (2,1)\}$ where $(1,0)$ is associated to the inner boundary component
and $(2,1)$ the outer boundary component.  If we consider only admissible cuts, then the 
AG-invariant of the resulting algebras will be given by only two ordered pairs, one per boundary
component.  In Figures \ref{fig ex (2,2)}--\ref{fig inner} we list the bound quivers given by the different possible admissible cuts grouped
by their AG-invariant. We use dashed lines to represent the zero relations induced by the cuts. 
There are two internal triangles, hence there are nine distinct admissible cuts.  If we make sure to cut 
both boundary components, we will get five algebras having the AG-invariant given by  $\{(2,2),(3,3)\}$.  
 It follows from \cite[Section 7]{AG} and Theorem \ref{thm AG}, that these are iterated tilted algebras of type $\tilde \AA_4$. 
\begin{figure}
 \begin{tikzpicture}
 {[rotate=-90]
\node at (1.75,-1.2) {$\cut z12 \cut v45$:};
\node[name=1] at (1,0) {1};
\node[name=2] at (2.5,0) {2};
\node[name=3] at (1.75,1)  {3};
\node[name=4] at (2.5,2) {4};
\node[name=5] at (1,2) {5};
\path[->]  (1) edge[bend left] (5) edge (3) 
	(5) edge (3)
	(3) edge (2) edge (4);
\path[dashed] (1) edge (2) (4) edge (5);
}
{[rotate=-90,xshift=2cm]
\node at (1.75,-1.2) {$\cut z23 \cut v45$:};
\node[name=1] at (1,0) {1};
\node[name=2] at (2.5,0) {2};
\node[name=3] at (1.75,1)  {3};
\node[name=4] at (2.5,2) {4};
\node[name=5] at (1,2) {5};
\path[->]  (1) edge[bend left] (5) edge (3)
	(5) edge (3)
	(3) edge (4)
	(2) edge (1);
\path[dashed] (2) edge (3) (4) edge (5);
}
{[rotate=-90,xshift=0cm, yshift=5cm]
\node at (1.75,-1.2) {$\cut z23 \cut w43$:};
\node[name=1] at (1,0) {1};
\node[name=2] at (2.5,0) {2};
\node[name=3] at (1.75,1)  {3};
\node[name=4] at (2.5,2) {4};
\node[name=5] at (1,2) {5};
\path[->]  (1) edge[bend left] (5) edge (3)
	(5) edge (3)
	(2) edge (1)
	(4) edge (5);
\path[dashed] (2) edge (3) (4) edge (3);
}
{[rotate=-90,xshift=2cm,yshift=5cm]
\node at (1.75,-1.2) {$\cut z12 \cut w43$:};
\node[name=1] at (1,0) {1};
\node[name=2] at (2.5,0) {2};
\node[name=3] at (1.75,1)  {3};
\node[name=4] at (2.5,2) {4};
\node[name=5] at (1,2) {5};
\path[->]  (1) edge[bend left] (5) edge (3) 
	(5) edge (3)
	(3) edge (2)
	(4) edge (5);
\path[dashed] (2) edge (1) (4) edge (3);
}
{[rotate=-90,xshift=4cm,yshift=2.5cm]
\node at (1.75,-1.2) {$\cut w13 \cut z35$:};
\node[name=1] at (1,0) {1};
\node[name=2] at (2.5,0) {2};
\node[name=3] at (1.75,1)  {3};
\node[name=4] at (2.5,2) {4};
\node[name=5] at (1,2) {5};
\path[->]  (1) edge[bend left] (5) 
	(4) edge (5)
	(3) edge (2) edge (4)
	(2) edge (1);
\path[dashed] (1) edge (3) (3) edge (5);
}
\end{tikzpicture}
\caption{The surface algebras of $(S,M,T)$ with AG-invariant given by $\{(2,2),(3,3)\}$}\label{fig ex (2,2)}
\end{figure}
The other four admissible cuts of $(S,M,T)$ split into two AG-classes: $\{(1,0),(4,5)\}$, corresponding to 
cutting the outer boundary component twice, and  $\{(3,4),(2,1)\}$ cutting the inner boundary component
 twice. See Figure~\ref{fig outer} and Figure~\ref{fig inner}, respectively.  Note that from 
 Corollary~\ref{cor annulus AG classes}, these are also the derived equivalence classes for the surface
 algebras of type $(S,M,T)$ coming from admissible cuts.

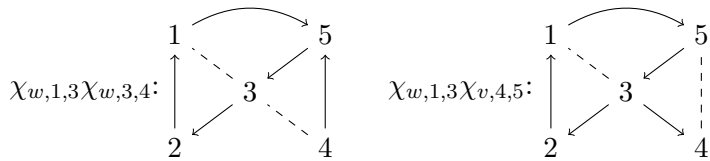
\begin{figure}
\begin{tikzpicture}
{[rotate=-90]
\node at (1.75,-1.2) {$\cut w13 \cut w34$:};
\node[name=1] at (1,0) {1};
\node[name=2] at (2.5,0) {2};
\node[name=3] at (1.75,1)  {3};
\node[name=4] at (2.5,2) {4};
\node[name=5] at (1,2) {5};
\path[->]  (1) edge[bend left] (5)
	(5) edge (3)
	(3) edge (2)
	(4) edge (5)
	(2) edge (1);
\path[dashed] (4) edge (3) (3) edge (1);
}
{[rotate=-90,yshift=5cm]
\node at (1.75,-1.2) {$\cut w13 \cut v45$:};
\node[name=1] at (1,0) {1};
\node[name=2] at (2.5,0) {2};
\node[name=3] at (1.75,1)  {3};
\node[name=4] at (2.5,2) {4};
\node[name=5] at (1,2) {5};
\path[->]  (1) edge[bend left] (5) 
	(5) edge (3)
	(3) edge (2) edge (4)
	(2) edge (1);
\path[dashed] (4) edge (5) (3) edge (1);
}
\end{tikzpicture}
\caption{The surface algebras  of $(S,M,T)$ with AG-invariant given by $\{(1,0),(4,5)\}$}\label{fig outer}
\end{figure}

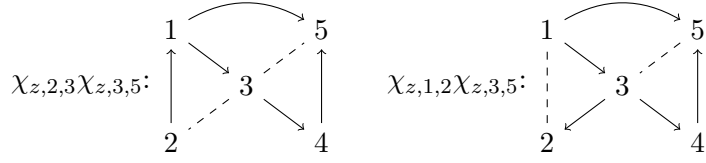
\begin{figure}
\begin{tikzpicture}
{[rotate=-90]
\node at (1.75,-1.2) {$\cut z23 \cut z35$:};
\node[name=1] at (1,0) {1};
\node[name=2] at (2.5,0) {2};
\node[name=3] at (1.75,1)  {3};
\node[name=4] at (2.5,2) {4};
\node[name=5] at (1,2) {5};
\path[->]  (1) edge[bend left] (5) edge (3)
	(3) edge (4)
	(4) edge (5)
	(2) edge (1);
\path[dashed] (3) edge (5) (3) edge (2);
}
{[rotate=-90,yshift=5cm]
\node at (1.75,-1.2) {$\cut z12 \cut z35$:};
\node[name=1] at (1,0) {1};
\node[name=2] at (2.5,0) {2};
\node[name=3] at (1.75,1)  {3};
\node[name=4] at (2.5,2) {4};
\node[name=5] at (1,2) {5};
\path[->]  (1) edge[bend left] (5) edge (3)
	(3) edge (4)
	(3) edge (2)
	(4) edge (5);
\path[dashed] (5) edge (3) (2) edge (1);
}
\end{tikzpicture}
\caption{The surface algebras with of $(S,M,T)$ AG-invariant given by $\{(3,4),(2,1)\}$}\label{fig inner}
\end{figure}
\end{ex}


\section{A geometric description of the module categories of surface algebras}\label{sect 5}
In this section, we study the module categories of surface algebras.
Let $(S,M,T)$ be a triangulated surface without punctures and let $(S,M\dag,T\dag)$ be a cut 
of $(S,M,T)$ with corresponding surface algebra $B_{T\dag}$. We will associate a category
 $\cale \dag$ to $(S,M\dag,T\dag)$ whose objects are given in terms of curves on the surface 
 $(S,M\dag)$, and we will show that there is a functor from $\cale\dag$ to $\cmod B_{T\dag}$  
 is faithful and dense in the category of string modules over $B_{T^\dd}$.  Consequently, the string modules over the surface algebra $B_{T\dag}$  
 have a geometric description in terms of certain curves on $(S,M\dag)$.

In the special case where $S$ is a disc, or more generally if $B_{T^\dd}$ is representation finite, we even have that $\cale\dag$ and $B_{T\dag}$ are equivalent categories.  
We will use this fact in Section~6 to give a geometric description of tilted algebras of type $\AA$ and 
their module categories. 

\subsection{Permissible arcs}

\begin{dfn}\mbox{}
\begin{enuma}
\item A curve $\gamma$ in $(S,M\dag,T\dag)$  is said to \df{consecutively cross} $\tau_1,\tau_2\in T\dag$
if $\gamma$ crosses $\tau_1$ and $\tau_2$ in the points $p_1$ and $p_2$, and the segment of $\gamma$ 
between the points $p_1$ and $p_2$ does not cross any other arc in $T\dag$.  
\item  A generalized arc or closed curve in $(S,M\dag)$ is called \df{permissible} if it is not consecutively 
crossing two non-adjacent sides of a \quasitri\ in $T\dag$. 
\item Two permissible generalized arcs $\gamma,\gamma'$ in $(S,M\dag)$ are called \df{equivalent} if
there is a side $\delta$ of a \quasitri\ $\triangle^\dd$ in $T\dag$ such that 
	\begin{enumi}
		\item $\gamma$ is homotopic to the concatenation of $\gamma'$ and $\delta$,
		\item both $\gamma$ and $\gamma'$ start at  an endpoint of $\delta$ and their first 
		crossing with $T\dag$ is with the side of $\triangle^\dd$ that is opposite to $\delta$.
	\end{enumi}
\end{enuma}
Examples of a non-permissible arc and of equivalent arcs are given in Figure~\ref{fig arcs ex}.
\end{dfn}

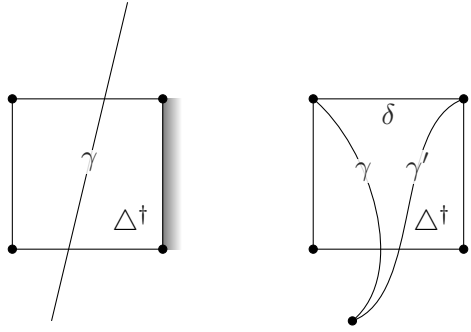
\begin{figure}
\begin{tikzpicture}
{[]
	\draw (0,0) node[pos=0,solid]  {}
		-- (2,0) node[pos=1,solid] {}
		-- (2,-2) node[pos=1,solid,label=135:$\triangle\dag$] {}
		-- (0,-2) node[pos=1,solid] {}
		-- cycle;
	{[on background layer] \shade[shading=axis,shading angle=90] (2,0)	rectangle (2.25,-2);}
	\draw (40:2cm) -- (-80:3cm) node[int] {$\gamma$};
}
{[xshift=4cm] 
	\draw (0,0) node[pos=0,solid]  {}
		-- (2,0) node[pos=1,solid] {} node[int,below,outer sep=1] {$\delta$}
		-- (2,-2) node[pos=1,solid,label=135:$\triangle\dag$] {}
		-- (0,-2) node[pos=1,solid] {}
		-- cycle;
	\draw (0,0) ..controls +(-40:1cm) and +(40:1cm) ..  (-80:3cm) node[int,pos=.35] {$\gamma$}
		   (2,0) ..controls +(200:1cm) and +(20:1cm) .. (-80:3cm) node[int,pos=.35] {$\gamma'$} 
		   node[pos=1,solid] {};
}
\end{tikzpicture}
\caption{Example of a non-permissible arc $\gamma$ on the left, and of two equivalent arcs 
$\gamma,\gamma'$ on the right}\label{fig arcs ex}
\end{figure}

Note that the arcs in $T\dag$ are permissible by definition.  Furthermore, note that the side $\delta$
 in part (c) of the definition may be, but does not have to be, a boundary segment.

\subsection{Pivots}
\begin{dfn} \label{def pivot}
Let $\gamma$ be a permissible generalized arc, and let $v$ and $w$ be its
endpoints. Let $f_v\gamma$ be the permissible generalized arc or boundary
segment obtained from $\gamma$ by fixing the endpoint $v$ and moving the
endpoint $w$ as follows:
\begin{enumi}
\item if $w$ is a vertex of a \quasitri\ $\triangle^\dd$ and $w'$ is the
	counterclockwise neighbor vertex of $w$ in $\triangle^\dd$ such that $\gamma$ is
	equivalent to an arc $\gamma'$ with endpoints $v,w'$ then move $w$ to the
	counterclockwise neighbor of $w'$ on the boundary of $S$.
	Thus $f_v\gamma$ is given by a curve homotopic to the concatenation of
	$\gamma'$ and the boundary segment that connects $w'$ to its counterclockwise
	neighbor. 
%
\item otherwise, move $w$ to its counterclockwise neighbor on the
	boundary. Thus $f_v\gamma$ is given by a curve homotopic to the concatenation
	of $\gamma$ and the boundary segment that connects $w$ to its counterclockwise
neighbor.
\end{enumi}
We call $f_v\gamma$ the \df{pivot of $\gamma$ fixing $v$}. 
\end{dfn}

\begin{remark}
In the special case of a trivial cut, that is $(S,M\dag,T\dag)=(S,M,T)$, the pivots are defined  
by condition (ii) only, since there are no \quasitri{s}.  In this case our pivots are the same as in \cite{BZ}, 
where it has been show that these pivots correspond to irreducible morphisms in $\cmod B_T$.
\end{remark}

\begin{dfn}
Let $\gamma$ be a permissible generalized arc in $(S,M\dag,T\dag)$  and let $v$ and $w$ be 
its endpoints.  Let $\tau^-\gamma$ be the permissible generalized arc given by 
$\tau^- \gamma = f_vf_w\gamma$. Dually, let $\tau^+\gamma$ be the permissible generalized 
arc in  $(S,M\dag,T\dag)$  such that $\tau^-(\tau^+ \gamma)=\gamma$. 
\end{dfn}
\begin{remark}
Again, in the case where $(S,M\dag,T\dag)=(S,M,T)$, it has been show in \cite{BZ} that 
$\tau^+$ and $\tau^-$ correspond to the Auslander-Reiten translation and the inverse
Auslander-Reiten translation in $\cmod B_T$, respectively.
\end{remark}
\subsection{The categories $\cale^\dd$ and $\cale$}
\begin{dfn}
Let $\cale\dag = \cale(S,M\dag,T\dag)$ be the additive category whose indecomposable
objects are the equivalence classes of permissible arcs in $(S,M\dag,T\dag)$ that are not
in $T\dag$ and whose morphisms between indecomposable objects are generated by the pivots
$f_v\in \Hom_{\cale\dag}(\gamma,f_v\gamma)$ subject to the relations $f_vf_w = f_wf_v$.
Here we use the convention that $f_v\gamma$ denotes the zero object in $\cale\dag$
whenever $f_v\gamma$ is a boundary segment in $(S, M^\dd)$ or an arc in $T\dag$.
\end{dfn}

Our first goal is to relate the category $\cale\dag =\cale(S,M\dag,T\dag)$ to the category
$\cale = \cale(S,M,T)$ of the original triangulated surface, by constructing a functor
$G\colon \cale\dag \to \cale$. Let $\gamma$ be an indecomposable object in $\cale\dag$;
thus $\gamma$ is represented by a permissible generalized arc in $(S,M\dag)$. The
generalized arc $\gamma$ is determined by the sequence of its crossing points with $T\dag$
together with the sequence of its segments between consecutive crossing points. Each of 
these segments lies entirely in a triangle or \quasitri\ of $T\dag$. Define $G(\gamma)$ to 
be the unique generalized arc in $(S,M)$ determined by the same sequence of crossing 
points and segments, with the difference that segments of $\gamma$ that were lying 
on a \quasitri\ become segments of $G(\gamma)$ which lie in the corresponding triangle.

If $f_v\colon \gamma \to f_v\gamma$ is a morphism in $\cale\dag$ given by  a pivot of type (ii) 
in Definition~\ref{def pivot}, then let $G(f_v)\colon G(\gamma) \to G(f_v\gamma)$ be the pivot 
in $\cale$ fixing the marked point corresponding to $v$ in $M$.  If $f_v\colon \gamma \to f_v\gamma$ 
is given by a pivot of type (i) in Definition~\ref{def pivot}, then let  
$G(f_v)\colon G(\gamma) \to G(f_v\gamma)$ be the minimal sequence of pivots in $\cale$, each pivot 
fixing the marked point corresponding to $v$ in $M$, which transforms $G(\gamma)$ into 
$G(f_v\gamma)$. 
\begin{prop}\label{prop:G full faithful}
$G$ is a full and faithful functor.
\end{prop}

\begin{proof}
The image of a composition of pivots under $G$ is equal to the composition of the images of the pivots.
Moreover, if $v\dag,w\dag$ are marked points in $M\dag$ and $v,w$ are the corresponding points in $M$, 
then $G(f_{v\dag})\circ G(f_{w\dag})(G(\gamma) )= G(f_{w\dag})\circ G(f_{v\dag})(G(\gamma))$, which shows that $G$ respects the 
relations on $\cale\dag$.

To show that $G$ is full, let $\gamma,\gamma'$ be indecomposable objects in $\cale\dag$ and let $f\in
\Hom_\cale(G(\gamma),G(\gamma'))$ be a nonzero morphism in $\cale$. Then $f$ is given by a sequence of
pivots $f=f_{v_i}\circ f_{v_{i-1}} \circ \cdots\circ f_{v_1}$ with $v_1,\cdots,v_i\in M$. Using the relations $f_vf_w =f_wf_v$
in $\cale$, we may suppose without loss of generality that 
$v_1=v_2=\cdots=v_h\neq v_{h+1} = \cdots = v_{i-1} = v_i$, for some $h$ with $1\leq h\leq i$. If all 
intermediate generalized arcs $f_{v_j}\circ f_{v_{j-1}} \circ \cdots\circ f_{v_1}(G(\gamma))$ are in the 
image of $G$, then there is a corresponding sequence of pivots 
$f_{v_i\dag}\circ f_{v_{i-1}\dag} \circ \cdots\circ f_{v_1\dag}$ in $\cale\dag$ such that 
$G(f_{v_i\dag}\circ f_{v_{i-1}\dag} \circ \cdots\circ f_{v_1\dag} (\gamma))
	= f_{v_i}\circ f_{v_{i-1}} \circ \cdots\circ f_{v_1}(G(\gamma))$. 
Otherwise, let $\alpha = f_{v_j}\circ f_{v_{j-1}} \circ \cdots\circ f_{v_1}(G(\gamma))$ be the first generalized 
arc in this  sequence that does not lie in the image of $G$. Then $\alpha$ must cross an internal triangle 
$\triangle$ in $T$ such that the corresponding (non-permissible) generalized arc $\za^\dd$ in $(S,M\dag,T\dag)$ 
crosses the corresponding \quasitri\ in two opposite sides consecutively. Let 
$\delta=f_{v_{j-1}} \circ f_{v_{j-2}} \circ \cdots \circ f_{v_1}(G(\gamma))$  be the immediate predecessor of 
$\alpha$ in the sequence, and let $\beta = f_{v_k} \circ f_{v_{k-1}} \circ \cdots \circ f_{v_1}(G(\gamma))$ 
with $k>j$ be the first arc after $\alpha$ in the sequence which is again in the image of $G$. We distinguish 
two cases.
\begin{enum}
\item Suppose that $v_j = v_{j+1} = \cdots = v_k$.  Thus $\beta$ and $\alpha$ are both incident to $v_j$,
	and it follows  that in  $\cale\dag$, we have $f_{v_j\dag}\delta\dag = \beta\dag$, where 
	$\delta\dag = f_{v_{j-1}\dag}\circ f_{v_{j-2}\dag}\circ \cdots \circ f_{v_1\dag}(\gamma)$ in $\cale\dag$ 
	is the object corresponding to $\delta$, and $\beta\dag$ the object in $\cale\dag$ corresponding 
	to $\beta$.  Clearly, 
	\[G( f_{v_j\dag}\circ f_{v_{j-1}\dag}\circ \cdots \circ f_{v_1\dag}(\gamma)) = G(\beta\dag) = \beta.\]
	By induction on the number of generalized arcs in the sequence  
	$f_{v_i}\circ f_{v_{i-1}} \circ \cdots\circ f_{v_1}(G(\gamma))$ which are not in the image of $G$, it 
	follows now that the morphism $f$ is in the image of $G$.  
\item Suppose that  $v_j\neq v_k$.  Then the triangle $\triangle$ separates the two arcs $G(\gamma)$ 
	and $\beta$,  which implies that none of the arcs in $T$ can cross both $G(\gamma)$ and 
	$\beta$.  Therefore $\Hom_\cale(G(\gamma),\beta) = 0$, hence $f=0$.
\end{enum}
This shows that $G$ is full. 

To prove that $G$ is faithful, it suffices to show that $G$ is faithful on pivots $\gamma\mapsto f_v\gamma$.  
For pivots of type (ii) in Definition~\ref{def pivot}, this is clear, so let $\gamma\mapsto f_v\gamma$ be 
a pivot of type (i).  Recall that such a pivot occurs if $\gamma$ is a permissible generalized arc in 
$(S,M\dag, T\dag)$ with endpoint $v,w\in M\dag$, where $w$ is a vertex of a \quasitri\ $\triangle^\dd$ in 
$T\dag$ such that  $\gamma$ is equivalent to a permissible generalized arc $\gamma'$ with endpoints 
$v,w'$, where $w'$  is the counterclockwise neighbor of $w$ in $\triangle^\dd$.  Since $\gamma$ and 
$\gamma'$ are equivalent, they define the same object of $\cale\dag$, so the morphism 
$\gamma\mapsto f_v\gamma$ can be represented by the pivot  $\gamma'\mapsto f_v\gamma'$ which 
is of type (ii).  This shows that the image of this morphism under $G$ is non-zero, hence $G$ is faithful.
\end{proof}

\subsection{Geometric description of module categories}
It has been shown in \cite{BZ} that there exists a faithful functor $F\colon \cale \to \cmod B_T$ which sends 
pivots to irreducible morphisms and $\tau^-$ to the inverse Auslander-Reiten translation in $\cmod B_T$. 
This functor induces a dense and faithful as  functor from $\cale$ to the category of $B_T$-string modules.

The composition $F\circ G$ is a faithful functor from $\cale\dag$ to $\cmod B_{T\dag}$.  We will now define 
another functor $\res \colon \cmod B_T\to \cmod B_{T\dag}$ and get a functor 
$H=\res\circ F \circ G\colon \cale\dag \to \cmod B_{T\dag}$ completing the following commutative
diagram.
\[\begin{tikzpicture}[scale=1.75]
\node[name=lt]  at (0,0) {$\cale\dag$};
\node[name=rt] at (2,0) {$\cale$};
\node[name=lb] at (0,-1) {$\cmod B_{T\dag}$};
\node[name=rb] at (2,-1) {$\cmod B_T$};
\path[->] 
	(lt) edge node[auto] {$H$} (lb) edge node[auto] {$G$} (rt)
	(rt) edge node[auto] {$F$} (rb)
	(rb) edge node[auto] {$\res$} (lb);
\end{tikzpicture}\]

Note that the inclusion map $i\colon B_{T\dag} \to B_T$ is an algebra homomorphism which sends 
$1_{B_{T\dag}}$ to $1_{B_T}$.  Thus $B_{T\dag}$ is a subalgebra of $B_T$, and we can define $\res$ to be the 
functor given by the restriction of scalars.

On the other hand, $B_{T\dag}$ is the quotient of $B_T$ by the two-sided ideal generated by the arrows in 
the cut.  Denoting  the projection homomorphism by $\pi\colon B_T \to B_{T\dag}$, we see that $\pi\circ i$ 
is the identity morphism on $B_{T\dag}$.  The extension of scalars functor 
$\iota = - \otimes_{B_{T\dag}} B_T\colon \cmod B_{T\dag} \to \cmod B_T$ is sending  $B_{T\dag}$-modules 
$M$ to $B_T$-modules $\iota(M)=M$ with the arrows in the cut acting trivially; and $\iota$ is the identity
on morphisms.  We have $\res\circ \iota = 1_{B_{T\dag}}$.

\begin{thm}\label{thm cats}Let $H=\res\circ F\circ G \colon \cale\dag \to \cmod B_{T\dag}$.
\begin{enuma}
\item The functor $H$ is faithful and induces a dense, faithful functor from $\cale\dag$ to
the category of string modules over $B_{T\dag}$.
\item If $f_v\colon \gamma \to f_v\gamma$ is a pivot in $\cale\dag$ then 
	$H(f_v):H(\gamma)\to H(f_v\gamma)$ is an irreducible morphism in $\cmod B_{T\dag}$.
\item The inverse Auslander-Reiten translate of $H(\gamma)$ is $H(\tau^-\gamma)$, and the Auslander-Reiten translate of $H(\gamma)$ is $H(\tau^+\gamma)$.
\item If $\gamma\in T\dag$ then $H(\tau^- \gamma)$ is an indecomposable projective 
	$B_{T\dag}$-module and $H(\tau^+\gamma)$ is an indecomposable injective $B_{T\dag}$-module.
\item If the surface $S$ is a disc, then $H$ is an equivalence of categories.
\item If the algebra $B_{T\dag}$ is of finite representation type, then $H$ is an equivalence of categories.
\end{enuma}
\end{thm}

\begin{proof}
Part (a).  The restriction of scalars functor is a faithful functor.  The fact that $F$ is faithful has been 
shown in \cite{BZ}, and $G$ is faithful by Proposition~\ref{prop:G full faithful}. Thus $H$ is faithful.

It also follows from \cite{BZ} that $B_T$-modules in the image of $F$ are string modules and that $F$ is 
a dense and faithful functor from $\cale$ to the category of string $B_T$-modules.  Consequently, the 
$B_{T\dag}$-modules in the image of $H$ are string modules.  Now let $M$ be any indecomposable string
module over $B_{T\dag}$, then $\iota(M)$ is a string module in $\cmod B_T$, and hence there exists
an indecomposable $\gamma\in \cale$ such that $F(\gamma)\cong\iota(M)$. Since the arrows in the 
cut act trivially on $\iota(M)$, the generalized arc $\gamma$ in $(S,M,T)$ lifts to a permissible 
generalized arc  $\gamma\dag$ in $(S,M\dag,T\dag)$.  Moreover 
$H(\gamma\dag) = \res\circ F(\gamma) \cong M$. Thus $H$ is dense to the category of string 
$B_{T\dag}$-modules.

Part (b). 
It has been shown in \cite{BR} that the irreducible morphisms starting at an indecomposable string module $M(w)$ can be described by adding hooks or deleting cohooks at the endpoints of  the string $w$. If $w$ is a string such that there exists an arrow $a$ such that $aw$ is a string, then let $v_a$ be the maximal non-zero path starting at $s(a)$ whose initial arrow is different from $a$.  Then there is an irreducible morphism $M(w)\to M(w_h)$, where $w_h=v_a^{-1}aw$ is obtained from $w$ by adding the ``hook" $v_a^{-1}a$ at the starting point of $w$. 
On the other hand, if $w$ is a string such that there is no arrow $a$ such that $aw $ is a string, then there are two possibilities: 
\begin{enumerate}
\item either $w$ contains no inverse arrow, in which case there is no irreducible morphism corresponding to the starting point of $w$,
\item or $w$ is of the form $w=u_a a^{-1} w_c$, where $a$ is an arrow and $u_a$ is maximal nonzero path ending at $t(a)$. In this case, there is an irreducible morphism $M(w)\to M(w_c)$ and $w_c$ is said to be obtained by deleting the ``cohook" $u_a a^{-1} $ at the starting point of the string $w$. 
\end{enumerate}
In a similar way, there are irreducible morphisms associated to adding a hook or deleting a cohook at the  terminal point of the string $w$.

The inverse Auslander-Reiten translation of the string module $M(w)$ corresponds to the string obtained from $w$ by performing the two operations of adding a hook, respectively deleting a cohook, at both endpoints of the string $w$.

Br\"ustle and Zhang have shown that for the Jacobian algebra $B_T$ of a triangulated unpunctured surface $(S,M,T)$, the irreducible morphisms between indecomposable string modules are precisely given by the pivots of the generalized arcs. Adapting their construction to cuts, we show now that the same is true for surface algebras. 

Let $\gamma^\dd$ be a generalized permissible arc in $(S,M^\dd,T^\dd)$ which is not in $T^\dd$, let $v^\dd$ and $w^\dd$ be the endpoints of $\gamma^\dd$ and consider the pivot $f_{v^\dd}\gamma^\dd$.
Applying the functor $H$, we obtain a homomorphism between $B_{T^\dd}$-string modules 
\[H(f_{v^\dd})\colon H(\gamma^\dd)\longrightarrow H (f_{v^\dd} \gamma^\dd),\]
and we must show that this morphism is irreducible.
Clearly, if the image under the extension of scalars functor $ \iota H (f_{v^\dd})$ is an irreducible morphism in $\cmod B_T$, then $H (f_{v^\dd} )$ is an irreducible morphism in $\cmod B_{T^\dd}$. By \cite{BZ}, this is precisely the case when $\iota H (f_{v^\dd} \gamma^\dd) = F( f_v\zg)$, where $f_v$ is the pivot in $(S,M,T)$ at the vertex $v$ corresponding to the vertex $v^\dd$, and $\zg=G(\gamma^\dd)$.

Thus we must show the result in the case where 
$\iota H (f_{v^\dd} \gamma^\dd) \ne F( f_v\zg)$. In this case, we have that $f_v\zg$ consecutively crosses two sides $ \alpha , \beta$ of a triangle $\triangle$ in $T$, which corresponds to a quasi-triangle $\triangle^\dd$ in $T^\dd$  in which the sides $\alpha, \beta$ give rise to two opposite sides $\alpha^\dd,\beta^\dd$, and $w$ is a common endpoint of $\alpha $ and $\beta$, which gives rise to  two points $w'$ and $w''$ in $\triangle^\dd$, see Figure \ref{fig irredmorph}.
\begin{figure}
\begin{tikzpicture}
{[]
  \node[solid,name=a] at (0,0) {};
  \node[solid,name=b] at (-2,-2) {};
  \node[solid,name=w,label=-20:$w$] at (2,-2) {};
  \shade[shading=axis,shading angle=90] (2,-2) rectangle (2.2,.5);
  \draw (a) -- (b)  
  	        -- (w) node[pos=.35,above] {$\triangle$}  node[int,pos=.2] {$\beta$}
  	        -- (a) node[int,pos=.8] {$\alpha$}
	   (2,-2) -- ++(205:3.5cm)
	   	   -- ++(0:4cm)
	   (2,-2) -- +(190:3cm)
  	   (2,-2) -- +(225:2.5cm) node[solid,name=v,label=-20:$v$] {} node[int,pos=.33] {$\gamma$}
  	   (2,-2) -- +(90:2.5cm) node[solid,name=u,label=45:$u$] {};
{[on background layer] \draw (v) -- (u) node[int,pos=.75] {$f_v\gamma$};}	
}
{[xshift=7cm,yshift=-2cm]
\shade[shading=axis,shading angle=90] (0,0) rectangle (.2,3.5);
\draw (0,0) -- (0,2)  
		node[name=w1,solid,label=-10:$w'$,pos=0] {}
  		node[name=w2,solid,label=-10:$w''$,pos=1] {}
	-- (0,3.5) 
		node[name=u,solid,label=10:$u\dag$,pos=1] {}
	(0,2) -- (-2.5,2)
		node[solid,pos=1,name=a] {}  
		node[int,pos=.6] {$\alpha\dag$}
	--(-2.5,0)
		node[solid,pos=1] {}
		node[pos=.6,right] {$\triangle\dag$}
	-- (0,0)
		node[int,pos=.25] {$\beta\dag$}
	-- ++(205:3cm)
	-- ++(0:4cm)
	(0,0) -- +(190:2cm)
  	(0,0)   -- +(225:2.5cm) node[solid,name=v,label=-20:$v\dag$] {} node[int,pos=.33] {$\gamma\dag$} ;	

	\draw (v) ..controls +(50:2cm) and +(225:1cm) .. (w2) node[left,pos=.75,fill=white,opacity=.8] {$f_{v\dag}\gamma\dag$}
		(v) .. controls +(50:2cm) and +(-35:1cm) .. (a) ;

}
\draw[->] (3,-1) -- (4,-1);
\end{tikzpicture}
\caption{Proof that irreducible morphisms are given by pivots}\label{fig irredmorph}
\end{figure}
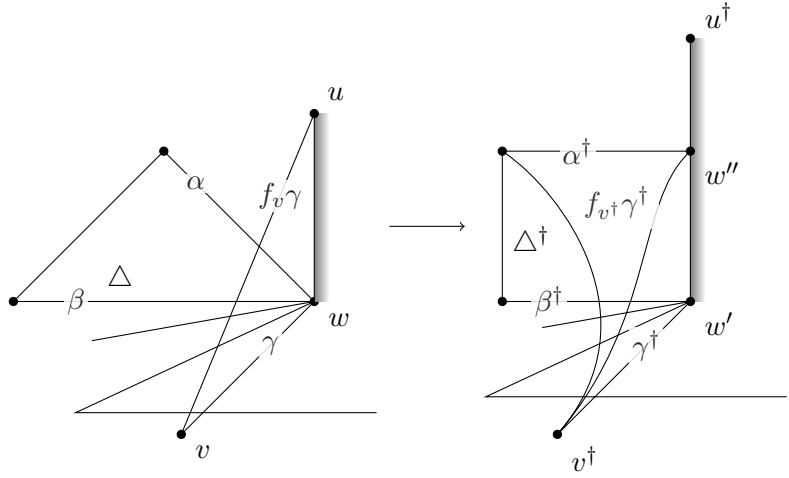
Now in this situation, $f_{v^\dd}\gamma^\dd$ is obtained from $\gamma$ by moving the endpoint $w'$ along the boundary segment of the quasi-triangle $\triangle^\dd$ to the point $w''$. The generalized arc $f_{v^\dd}\gamma^\dd$ crossed every arc in $T^\dd$ that is crossed by $\gamma^\dd$ and, in addition, $f_{v^\dd}\gamma^\dd$ also crosses every arc in $T^\dd$ that is in the complete fan $\beta_1,\beta_2,\ldots,\beta_j=\beta^\dd$ at $w^\dd$. It particular, its last crossing is with the arc $\beta^\dd$. On the level of the string modules, this corresponds to obtaining the string modules $H(f_{v^\dd}\gamma^\dd)$  by adding the hook $\bullet\ot\beta_1\to\beta_2\to \cdots\to \beta^\dd$ to the string module $H(\gamma^\dd)$. Thus the morphism $H(f_{v^\dd})$ is irreducible.

Part (c). This follows directly from (b) and  from the description of the Auslander-Reiten translation for string modules in \cite{BR}. 

Part (d). The statement about the projective modules follows from (c) and the fact that the indecomposable projective modules are the indecomposable string modules $M(w)$ whose Auslander-Reiten translate is zero. Similarly, the statement about the injective modules follows from (c) and the fact that the indecomposable injective modules are the indecomposable string modules $M(w)$ whose inverse Auslander-Reiten translate is zero.

Part (e). 
If $S$ is a disc, then it has been shown in \cite{CCS} that the functor $F$ is an equivalence of 
categories.  In particular, all $B_T$-modules, and therefore all $B_{T\dag}$-modules, are string modules, 
which shows that  $H\colon \cale\dag\to \cmod B_T$ is dense.

It remains to show that $H$ is full.  Let $\gamma,\gamma'$ be indecomposable objects in $\cale\dag$ and
let $f\in \Hom_{B_{T\dag}}(H(\gamma),H(\gamma'))$.  Applying the functor $\iota$ yields 
$\iota(f)\in \Hom_{B_{T}}(\iota\circ H(\gamma),\iota\circ H(\gamma'))$.  Since $\gamma$ and $\gamma'$ 
are objects in $\cale\dag$, their images under $F\circ G$ are $B_T$-modules on which the arrows in the cut act 
trivially.  Thus, applying $\iota\circ\res$ to $F\circ G(\gamma)$ and $F\circ G(\gamma')$ gives back 
$F\circ G(\gamma)$ and $F\circ G(\gamma')$, hence $\iota\circ H(\gamma) = F\circ G(\gamma)$ and 
$\iota\circ H(\gamma')=F\circ G(\gamma')$. Since the composition $F\circ G$ is full, it follows that there 
exists a morphism $f\dag\in \Hom_{\cale\dag}(\gamma,\gamma')$ such that $F\circ G(f\dag)=\iota(f)$.  
Applying $\res$ now yields $H(f\dag)=f$.  Thus $H$ is full.

Part (f). 
If $B_{T\dag}$ is of finite representation type, then all $B_{T\dag}$-modules are
string modules. Indeed, a band module would give rise to infinitely many isoclasses of
indecomposable band modules. Thus $H$ is dense by part (a). Since $B_{T\dag}$ is of finite
representation type it follows that the Auslander-Reiten quiver of $B_{T\dag}$ has only one
connected component and that each morphism between indecomposable $B_{T\dag}$-modules is
given as a composition of finitely many irreducible morphisms. It follows from part (b)
that $H$ is full.
\end{proof}

\begin{remark} 
Since $B_{T\dag}$ is a string algebra, it has two types of indecomposable modules: string
modules and band modules, see \cite{BR}. The theorem shows that the string modules are given by the equivalence
classes of permissible generalized arcs.

 The band modules can be parameterized by triples 
$(\gamma,\ell,\phi)$ where $\gamma$ is a permissible closed curve corresponding to a
generator of the fundamental group of $S$, $\ell$ is a positive integer, and $\phi$ is an automorphism of $k^\ell$.
The Auslander-Reiten structure for band modules is very simple: the irreducible morphisms are of the form 
$M(\zg,\ell,\phi) \to M(\zg,\ell+1,\phi)$  and, if $\ell>1$, $M(\zg,\ell,\phi) \to M(\zg,\ell-1,\phi)$; whereas the Auslander-Reiten translation is the identity $\tau M(\zg,\ell,\phi) = M(\zg,\ell,\phi)$.
\end{remark}
\subsection{Finite representation type}
\begin{cor}\label{cor: finite type 1}
The surface algebra $B_{T\dag}$ is of finite representation type if and only if no simple
closed non-contractible curve in $(S,M\dag,T\dag)$ is permissible.
\end{cor}
\begin{proof} Every permissible non-contractible closed curve defines infinitely many 
indecomposable band modules.  Conversely, if no simple non-contractible closed curve
is permissible, then $B_{T\dag}$ has no band modules.  It follows that $B_{T\dag}$ is
of finite representation type. 
\end{proof}

\begin{cor}\label{cor: finite type 2}\mbox{}
\begin{enuma}
\item If $S$ is a disc, then $B_{T\dag}$ is of finite representation type.
\item If $S$ is an annulus, then $B_{T\dag}$ is of finite representation type if and only if
there is a \quasitri\ in $T\dag$ with two vertices on one boundary component and two 
vertices on the other boundary component. 
\end{enuma}
\end{cor}

\begin{proof}
Part (a).  If $S$ is a disc, then there are no simple closed non-contractible curves in $S$, and the
result follows from Corollary~\ref{cor: finite type 1}.

Part (b).  The fundamental group of $S$ is generated by a closed curve $\gamma$ that goes around 
the inner boundary component exactly once.  This curve $\gamma$ is permissible precisely when 
it does not cross two opposite sides of a \quasitri\ in $T\dag$, and this is precisely the case if there
is there is no \quasitri\ with two vertices on the interior boundary component and two vertices on the 
exterior boundary component.
\end{proof}

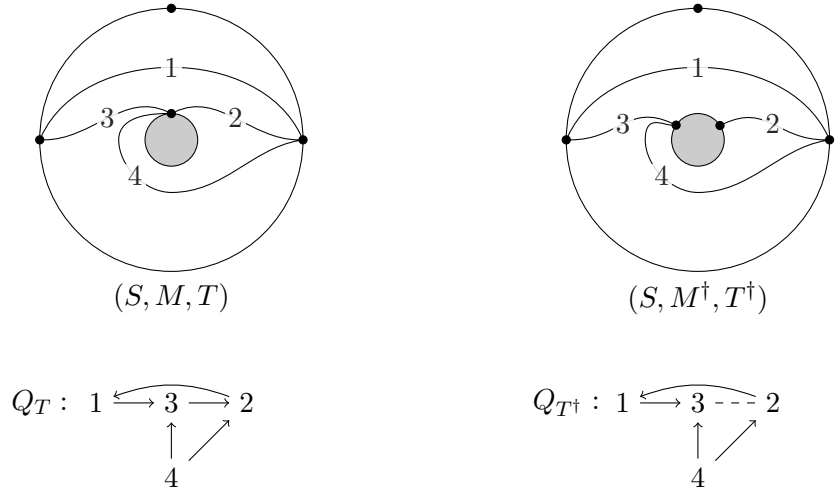
\begin{figure}
\begin{tikzpicture}
{[scale=.35]
	\draw (0,0) circle[radius=5cm] ;
	\filldraw[fill=black!20] (0,0) circle[radius=1cm];
	\node[solid,name=a] at (0:5cm) {};
	\node[solid,name=b] at (90:5cm) {};
	\node[solid,name=c] at (180:5cm) {};
	\node[solid,name=d] at (90:1cm) {};
	
	\draw (a) .. controls +(115:4cm) and +(65:4cm) .. (c) node[int] {1}
		   (a)  .. controls +(180:2cm) and +(25:2cm) .. (d) node[int] {2}
		   (c)  .. controls +(0:2cm) and +(155:2cm) .. (d) node[int] {3}
		   (d) .. controls +(180:.5cm) and +(90:1cm) .. (180:2cm)
		   (180:2cm) .. controls +(-90:1cm) and +(-180:1cm) .. (270:2cm) node[int] {4}
		   (270:2cm) .. controls +(0:2cm) and +(190:2cm) .. (a);
	\node at (270:6cm) {$(S,M,T)$};
}
{[yshift=-3.5cm,xshift=-2cm]
	\node at (.25,0) {$Q_T:$};
	\node[name=1] at (1,0) {$1$};
	\node[name=3] at (2,0) {$3$};
	\node[name=2] at (3,0) {$2$};
	\node[name=4] at (2,-1) {$4$};
	\path[->] (1) edge (3) 
			(3) edge(2)
			(2) edge[bend right=20]  (1)
			(4) edge (3) edge (2);
}
{[xshift=7cm,scale=.35]
	\draw (0,0) circle[radius=5cm];
	\filldraw[fill=black!20] (0,0) circle[radius=1cm];
	\node[solid,name=a] at (0:5cm) {};
	\node[solid,name=b] at (90:5cm) {};
	\node[solid,name=c] at (180:5cm) {};
	\node[solid,name=d] at (146:1cm) {};
	\node[solid,name=d'] at (33:1cm) {};
	
	\draw (a) .. controls +(115:4cm) and +(65:4cm) .. (c) node[int] {1}
		   (a)  .. controls +(180:2cm) and +(25:2cm) .. (d') node[int] {2}
		   (c)  .. controls +(0:2cm) and +(155:2cm) .. (d) node[int] {3}
		   (d) .. controls +(180:.5cm) and +(90:1cm) .. (180:2cm)
		   (180:2cm) .. controls +(-90:1cm) and +(-180:1cm) .. (270:2cm) node[int] {4}
		   (270:2cm) .. controls +(0:2cm) and +(190:2cm) .. (a);
	\node at (270:6cm) {$(S,M\dag,T\dag)$};
}
{[yshift=-3.5cm,xshift=5cm]
	\node at (.25,0) {$Q_{T\dag}:$};
	\node[name=1] at (1,0) {$1$};
	\node[name=3] at (2,0) {$3$};
	\node[name=2] at (3,0) {$2$};
	\node[name=4] at (2,-1) {$4$};
	\path[->] (1) edge (3) 
			(3) edge[dashed,-] (2)
			(2) edge[bend right=20]  (1)
			(4) edge (3) edge (2);
}
\end{tikzpicture}
\caption{An admissible cut of the annulus whose corresponding surface algebra is of finite representation type}\label{fig:annulus finite type ex}
\end{figure}

\begin{ex}
Let $(S,M,T)$ and $(S,M\dag,T\dag)$ be as in Figure~\ref{fig:annulus finite type ex}.  
By Corollary~\ref{cor: finite type 2}, $B_{T\dag}$ is of finite representation type.  The 
Auslander-Reiten quiver is given in Figure~\ref{fig: AR ex}, where modules
are given by their Loewy series in the upper picture and by the 
corresponding generalized permissible arc on the surface in the lower picture. 
\end{ex}
\begin{figure}
\begin{tikzpicture}[pin distance=.75,scale=1.4]
{[xscale=1.2]
\foreach \x/\y/\d [count=\n] in 
	{1/1/1, 
	 2/2/\dimv21, 2/4/3,
	 3/1/2,3/3/\dimvv4{3\ 2}{\ds 1}, 3/5/\dimv13,
	 4/2/\dimv{4}{3\ 2}, 4/4/\dimvv{1\ 4}{\ms3 2}{\ds\ 1},
	 5/1/\dimv43, 5/3/\dimv{1\ 4}{\ 3\ 2}, 5/5/\dimvv421,
	 6/2/\dimv{1\ 4}{\ms 3}, 6/4/\dimv42,
	7/3/4, 7/1/1
	}
{
	\node[name=\n] at (\x,\y) {$\d$};
}
\path[->] (1) edge (2) 
	       (2) edge (4) edge (5)
	       (3) edge (5) edge (6)
	       (4) edge (7)
	       (5) edge (7) edge (8)
	       (6) edge (8)
	       (7) edge (9) edge (10)
	       (8) edge (10)  edge (11)
	       (9) edge (12)
	       (10) edge (12) edge (13)
	       (11) edge (13)
	       (12) edge (14) edge (15)
	       (13) edge (14);
}
\end{tikzpicture} \\
\vspace{1.5cm}
\begin{tikzpicture}[scale=.8]
	{[xshift=0,yshift=0,scale=.15]
	\node[name=1,outer sep=1cm] at (0,0) {};
	\fill[fill=black!20] (0,0) circle[radius=1cm];
	\draw (0,0) circle[radius=5cm] (0,0) circle[radius= 1cm] 
		 node[solid] at (0:5cm) {} node[solid] at (90:5cm) {} node[solid] at (180:5cm) {}
		 node[solid] at (146:1cm) {} node[solid] at (33:1cm) {}
		 (146:1cm) -- (90:5cm);
	}
	{[xshift=2cm,yshift=2cm,scale=.15]
	\node[name=2,outer sep=1cm] at (0,0) {};
	\fill[fill=black!20] (0,0) circle[radius=1cm];
	\draw (0,0) circle[radius=5cm] (0,0) circle[radius= 1cm] 
		 node[solid] at (0:5cm) {} node[solid] at (90:5cm) {} node[solid] at (180:5cm) {}
		 node[solid] at (146:1cm) {} node[solid] at (33:1cm) {}
		 (146:1cm) .. controls +(180:.5cm) and +(90:1cm) .. (180:2cm)
		 (180:2cm) .. controls +(-90:1cm) and +(-180:1cm) .. (270:2cm)
		 (270:2cm) .. controls +(0:1cm) and +(-90:1cm) .. (0:2cm)
		 (0:2cm) -- (90:5cm);
	}
	{[xshift=2cm,yshift=6cm,scale=.15]
	\node[name=3,outer sep=1cm] at (0,0) {};
	\filldraw[fill=black!20] (0,0) circle[radius=1cm];
	\draw (0,0) circle[radius=5cm] (0,0) circle[radius= 1cm] 
		 node[solid] at (0:5cm) {} node[solid] at (90:5cm) {} node[solid] at (180:5cm) {}
		 node[solid] at (146:1cm) {} node[solid] at (33:1cm) {}
		  (0:5cm) .. controls +(180:1cm) and +(0:1cm) .. (90:2cm)
		 (90:2cm) ..controls +(180:1cm) and +(90:1cm) .. (180:2cm)
		 (180:2cm) .. controls +(-90:1cm) and +(-180:1cm) .. (270:2cm)
		 (270:2cm) ..controls +(0:1cm) and +(180:1cm) .. (0:5cm);
	}
	{[xshift=4cm,yshift=0cm,scale=.15]
	\node[name=4,outer sep=1cm] at (0,0) {};
	\fill[fill=black!20] (0,0) circle[radius=1cm];
	\draw (0,0) circle[radius=5cm] (0,0) circle[radius= 1cm] 
		 node[solid] at (0:5cm) {} node[solid] at (90:5cm) {} node[solid] at (180:5cm) {}
		 node[solid] at (146:1cm) {} node[solid] at (33:1cm) {}
		  (146:1cm) .. controls +(180:.5cm) and +(90:1cm) .. (180:2cm)
		 (180:2cm) .. controls +(-90:1cm) and +(-180:1cm) .. (270:2cm)
		 (270:2cm) .. controls +(0:1cm) and +(-90:1cm) .. (0:2cm)
		 (0:2cm) .. controls +(90:1cm) and +(0:1cm) .. (90:2cm)
		 (90:2cm) ..controls +(180:1cm) and +(0:1cm) .. (180:5cm);
	}
	{[xshift=4cm,yshift=4cm,scale=.15]
	\node[name=5,outer sep=1cm] at (0,0) {};
	\fill[fill=black!20] (0,0) circle[radius=1cm];
	\draw (0,0) circle[radius=5cm] (0,0) circle[radius= 1cm] 
		 node[solid] at (0:5cm) {} node[solid] at (90:5cm) {} node[solid] at (180:5cm) {}
		 node[solid] at (146:1cm) {} node[solid] at (33:1cm) {}
		 (0:5cm) .. controls +(180:1cm) and +(0:1cm) .. (90:2cm)
		 (90:2cm) ..controls +(180:1cm) and +(90:1cm) .. (180:2cm)
		 (180:2cm) .. controls +(-90:1cm) and +(-180:1cm) .. (270:2cm)
		 (270:2cm) .. controls +(0:1cm) and +(-90:1cm) .. (0:2cm)
		 (0:2cm) -- (90:5cm);
	}
	{[xshift=4cm,yshift=8cm,scale=.15]
	\node[name=6,outer sep=1cm] at (0,0) {};
	\fill[fill=black!20] (0,0) circle[radius=1cm];
	\draw (0,0) circle[radius=5cm] (0,0) circle[radius= 1cm] 
		 node[solid] at (0:5cm) {} node[solid] at (90:5cm) {} node[solid] at (180:5cm) {}
		 node[solid] at (146:1cm) {} node[solid] at (33:1cm) {}
		(90:5cm) ..controls +(270:1cm) and +(90:1cm) .. (180:2cm)
		(180:2cm) ..controls +(270:1cm)  and +(180:1cm) .. (270:2cm)
		 (270:2cm) ..controls +(0:1cm) and +(180:1cm) .. (0:5cm);
	}
	{[xshift=6cm,yshift=2cm,scale=.15]
	\node[name=7,outer sep=1cm] at (0,0) {};
	\fill[fill=black!20] (0,0) circle[radius=1cm];
	\draw (0,0) circle[radius=5cm] (0,0) circle[radius= 1cm] 
		 node[solid] at (0:5cm) {} node[solid] at (90:5cm) {} node[solid] at (180:5cm) {}
		 node[solid] at (146:1cm) {} node[solid] at (33:1cm) {}
		 (0:5cm) .. controls +(180:1cm) and +(0:1cm) .. (90:2.5cm)
		 (90:2.5cm) ..controls +(180:1cm) and +(90:1cm) .. (180:2cm)
		 (180:2cm) .. controls +(-90:1cm) and +(-180:1cm) .. (270:2cm)
		 (270:2cm) .. controls +(0:1cm) and +(-90:1cm) .. (0:2cm)
		 (0:2cm) .. controls +(90:1cm) and +(0:1cm) .. (90:2cm)
		 (90:2cm) ..controls +(180:1cm) and +(0:1cm) .. (180:5cm);
	}
	{[xshift=6cm,yshift=6cm,scale=.15]
	\node[name=8,outer sep=1cm] at (0,0) {};
	\fill[fill=black!20] (0,0) circle[radius=1cm];
	\draw (0,0) circle[radius=5cm] (0,0) circle[radius= 1cm] 
		 node[solid] at (0:5cm) {} node[solid] at (90:5cm) {} node[solid] at (180:5cm) {}
		 node[solid] at (146:1cm) {} node[solid] at (33:1cm) {}
		(90:5cm) ..controls +(270:1cm) and +(90:1cm) .. (180:2cm)
		 (180:2cm) .. controls +(-90:1cm) and +(-180:1cm) .. (270:2cm)
		 (270:2cm) .. controls +(0:1cm) and +(-90:1cm) .. (0:2cm)
		 (0:2cm) -- (90:5cm);
	}
	{[xshift=8cm,yshift=0cm,scale=.15]
	\node[name=9,outer sep=1cm] at (0,0) {};
	\fill[fill=black!20] (0,0) circle[radius=1cm];
	\draw (0,0) circle[radius=5cm] (0,0) circle[radius= 1cm] 
		 node[solid] at (0:5cm) {} node[solid] at (90:5cm) {} node[solid] at (180:5cm) {}
		 node[solid] at (146:1cm) {} node[solid] at (33:1cm) {}
		  (0:5cm) .. controls +(180:1cm) and +(0:1cm) .. (90:2cm)
		 (90:2cm) ..controls +(180:1cm) and +(90:1cm) .. (180:2cm)
		 (180:2cm) .. controls +(-90:1cm) and +(-180:1cm) .. (270:2cm)
		 (270:2cm) .. controls +(0:1cm) and +(-90:1cm) .. (0:2cm)
		 (0:2cm) ..controls +(90:1cm) and +(0:1cm) .. (33:1cm);
	}
	{[xshift=8cm,yshift=4cm,scale=.15]
	\node[name=10,outer sep=1cm] at (0,0) {};
	\draw (0,0) circle[radius=5cm] (0,0) circle[radius= 1cm] 
		 node[solid] at (0:5cm) {} node[solid] at (90:5cm) {} node[solid] at (180:5cm) {}
		 node[solid] at (146:1cm) {} node[solid] at (33:1cm) {}
		(90:5cm) ..controls +(270:1cm) and +(90:1cm) .. (180:2cm)
		 (180:2cm) .. controls +(-90:1cm) and +(-180:1cm) .. (270:2cm)
		 (270:2cm) .. controls +(0:1cm) and +(-90:1cm) .. (0:2cm)
		 (0:2cm) .. controls +(90:1cm) and +(0:1cm) .. (90:2cm)
		 (90:2cm) ..controls +(180:1cm) and +(0:1cm) .. (180:5cm);
	}
	{[xshift=8cm,yshift=8cm,scale=.15]
	\node[name=11,outer sep=1cm] at (0,0) {};
	\fill[fill=black!20] (0,0) circle[radius=1cm];
	\draw (0,0) circle[radius=5cm] (0,0) circle[radius= 1cm] 
		 node[solid] at (0:5cm) {} node[solid] at (90:5cm) {} node[solid] at (180:5cm) {}
		 node[solid] at (146:1cm) {} node[solid] at (33:1cm) {}
		(180:5cm) ..controls +(0:1cm) and +(180:1cm) .. (270:2cm)
		 (270:2cm) .. controls +(0:1cm) and +(-90:1cm) .. (0:2cm)
		 (0:2cm) .. controls +(90:1cm) and +(270:1cm) .. (90:5cm);
	}
	{[xshift=10cm,yshift=2cm,scale=.15]
	\node[name=12,outer sep=1cm] at (0,0) {};
	\fill[fill=black!20] (0,0) circle[radius=1cm];
	\draw (0,0) circle[radius=5cm] (0,0) circle[radius= 1cm] 
		 node[solid] at (0:5cm) {} node[solid] at (90:5cm) {} node[solid] at (180:5cm) {}
		 node[solid] at (146:1cm) {} node[solid] at (33:1cm) {}
		 (33:1cm) ..controls +(45:1cm) and +(90:1cm) .. (0:2cm)
		 (0:2cm) .. controls +(270:1cm) and +(0:1cm) .. (270:2cm)
		 (270:2cm) .. controls +(180:1cm) and +(270:1cm) .. (180:2cm)
		 (180:2cm) .. controls +(90:1cm) and +(270:1cm) .. (90:5cm);
	}
	{[xshift=10cm,yshift=6cm,scale=.15]
	\node[name=13,outer sep=1cm] at (0,0) {};
	\fill[fill=black!20] (0,0) circle[radius=1cm];
	\draw (0,0) circle[radius=5cm] (0,0) circle[radius= 1cm] 
		 node[solid] at (0:5cm) {} node[solid] at (90:5cm) {} node[solid] at (180:5cm) {}
		 node[solid] at (146:1cm) {} node[solid] at (33:1cm) {}
		(180:5cm) ..controls +(0:1cm) and +(180:1cm) .. (270:2cm)
		 (270:2cm) .. controls +(0:1cm) and +(-90:1cm) .. (0:2cm)
		 (0:2cm) .. controls +(90:1cm) and +(0:1cm) .. (90:2cm)
		 (90:2cm) ..controls +(180:1cm) and +(0:1cm) .. (180:5cm);
	}
	{[xshift=12cm,yshift=4cm,scale=.15]
	\node[name=14,outer sep=1cm] at (0,0) {};
	\fill[fill=black!20] (0,0) circle[radius=1cm];
	\draw (0,0) circle[radius=5cm] (0,0) circle[radius= 1cm] 
		 node[solid] at (0:5cm) {} node[solid] at (90:5cm) {} node[solid] at (180:5cm) {}
		 node[solid] at (146:1cm) {} node[solid] at (33:1cm) {}
		 (33:1cm) ..controls +(45:1cm) and +(90:1cm) .. (0:2cm)
		 (0:2cm) .. controls +(270:1cm) and +(0:1cm) .. (270:2cm)
		 (270:2cm) .. controls +(180:1cm) and +(0:1cm) .. (180:5cm);
	}
	{[xshift=12cm,yshift=0,scale=.15]
	\node[name=15,outer sep=1cm] at (0,0) {};
	\fill[fill=black!20] (0,0) circle[radius=1cm];
	\draw (0,0) circle[radius=5cm] (0,0) circle[radius= 1cm] 
		 node[solid] at (0:5cm) {} node[solid] at (90:5cm) {} node[solid] at (180:5cm) {}
		 node[solid] at (146:1cm) {} node[solid] at (33:1cm) {}
		 (146:1cm) -- (90:5cm);
	}
\path[<-] (1) edge (2) 
	       (2) edge (4) edge (5)
	       (3) edge (5) edge (6)
	       (4) edge (7)
	       (5) edge (7) edge (8)
	       (6) edge (8)
	       (7) edge (9) edge (10)
	       (8) edge (10)  edge (11)
	       (9) edge (12)
	       (10) edge (12) edge (13)
	       (11) edge (13)
	       (12) edge (14) edge (15)
	       (13) edge (14);
\end{tikzpicture}
\caption{The Auslander-Reiten quiver of the surface algebra in Figure \ref {fig:annulus finite type ex}. Indecomposable modules are represented by their Loewy series in the upper diagram, and by their permissible generalized arc in the lower diagram.}\label{fig: AR ex}
\end{figure}
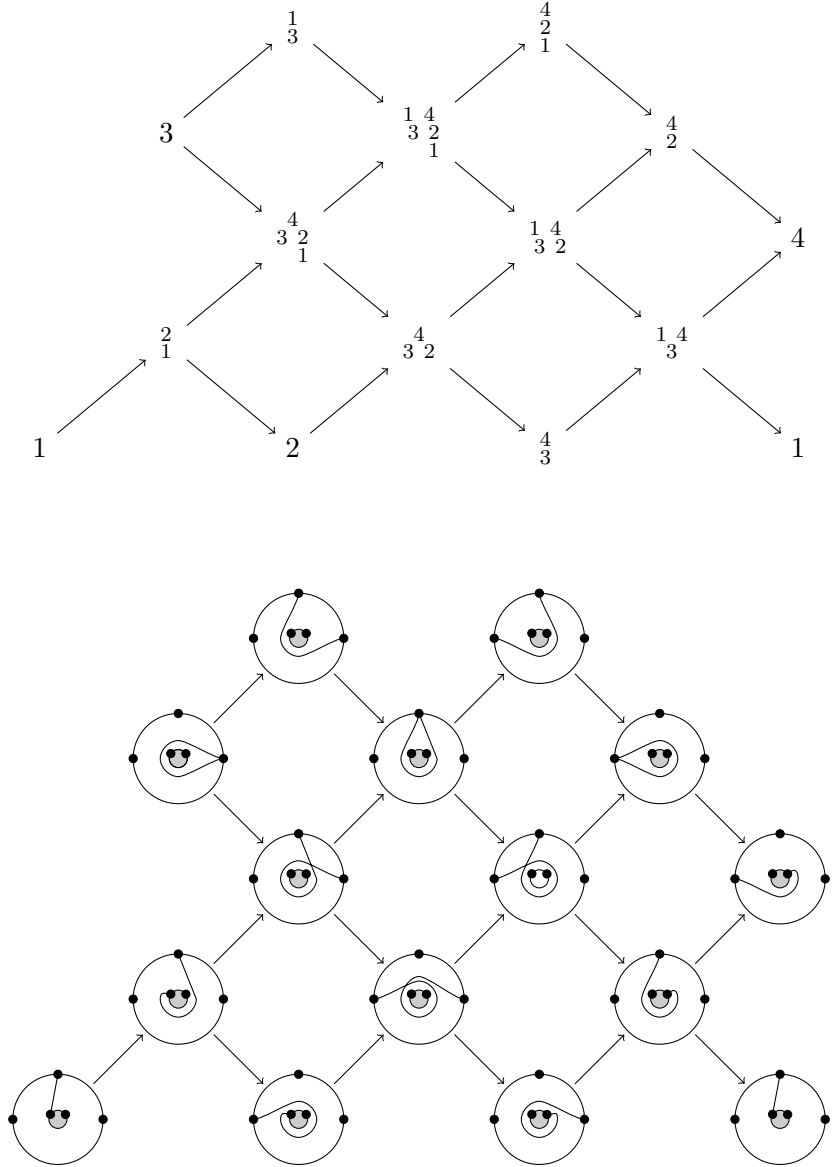

\section{ A geometric description of iterated tilted algebras of type $\AA_n$ of global dimension at most 2}\label{sect 6}

In this section, we apply the results in section 3 to obtain a description of the admissible cut 
surface algebras of type $(S,M,T)$ where $S$ is a disc.  

In this case, the Jacobian algebra $B_T$ is a cluster-tilted algebra of type $\AA$ and $B_{T\dag}$ is 
a quotient of $B_T$ by an admissible cut.  It has been show in \cite{BFPPT} that the algebras
$B_{T\dag}$ obtained in this way are precisely the iterated tilted algebras of type $\AA$ of 
global dimension at most two.

\begin{prop}
Every iterated tilted algebra $C$ of type $\AA_n$ of global dimension at most two is 
isomorphic to the endomorphism algebra of a partial cluster-tilting object in the cluster
category of type $\AA_{n+\ell}$, where $\ell$ is the number of relations in a minimal
system of relations for $C$.
\end{prop}

\begin{proof}
Let $(S,M\dag,T\dag)$ be the admissible cut corresponding to $C$. Then 
$C\cong \End_{\calc{(S,M\dag)}}(T\dag)$
and $(S,M\dag)$ is a disc with $n+\ell$ marked points.
\end{proof}

\begin{ex} Let $C$ be the tilted algebra given by the bound quiver
\begin{tikzpicture}[baseline=-3]
\node[name=1] at (1,0) {1};
\node[name=2] at (2,0) {2};
\node[name=3] at (3,0) {3};
\node[name=4] at (4,0) {4};
\path[->] (1) edge (2) (2) edge (3) (3) edge (4) (1) edge[bend left=20,dashed,-] (3);
\end{tikzpicture}.
Then we have $n=4$ and $\ell=1$, so $C$ is the endomorphism algebra of a partial cluster-tilting object in the cluster category of type $\AA_5$. By \cite{CCS}  this category can be seen as the category of diagonals in an octagon, and the partial cluster-tilting object corresponds to the following partial triangulation.
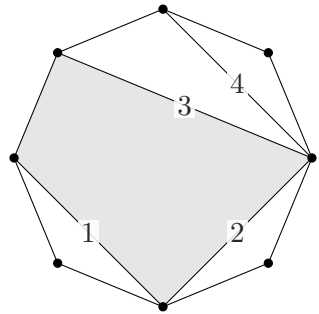
\[\begin{tikzpicture}[scale=.66]
\foreach \a [count=\n] in {0,1,2,3,4,5,6,7}
	\node[name=\n,solid] at (360*\a/8:3cm) {};
\foreach \x [remember=\x as \lastx (initially 8)] in {1,...,8}
	\draw (\lastx.center) -- (\x.center);
\draw (1) -- (3) node[int] {4}
	(1) -- (4) node[int] {3}
	(1) -- (7) node[int] {2}
	(7) -- (5) node[int] {1};
{[on background layer] \fill[black!10] (1.center) -- (4.center) --(5.center) -- (7.center) -- cycle;} 
\end{tikzpicture}\]
Here the shaded region indicates the unique quasi-triangle. The Auslander-Reiten quiver is described in Figure~\ref{fig tilted}. Note that permissible arcs are not allowed to consecutively cross the two sides label $1$ and $3$ of the shaded quasi-triangle. Also note that   each of the modules $\begin{array}{c} 3 \vspace{-2pt}\\ 4\end{array}$, $3,2$ and $1$, can be represented by two equivalent permissible arcs.

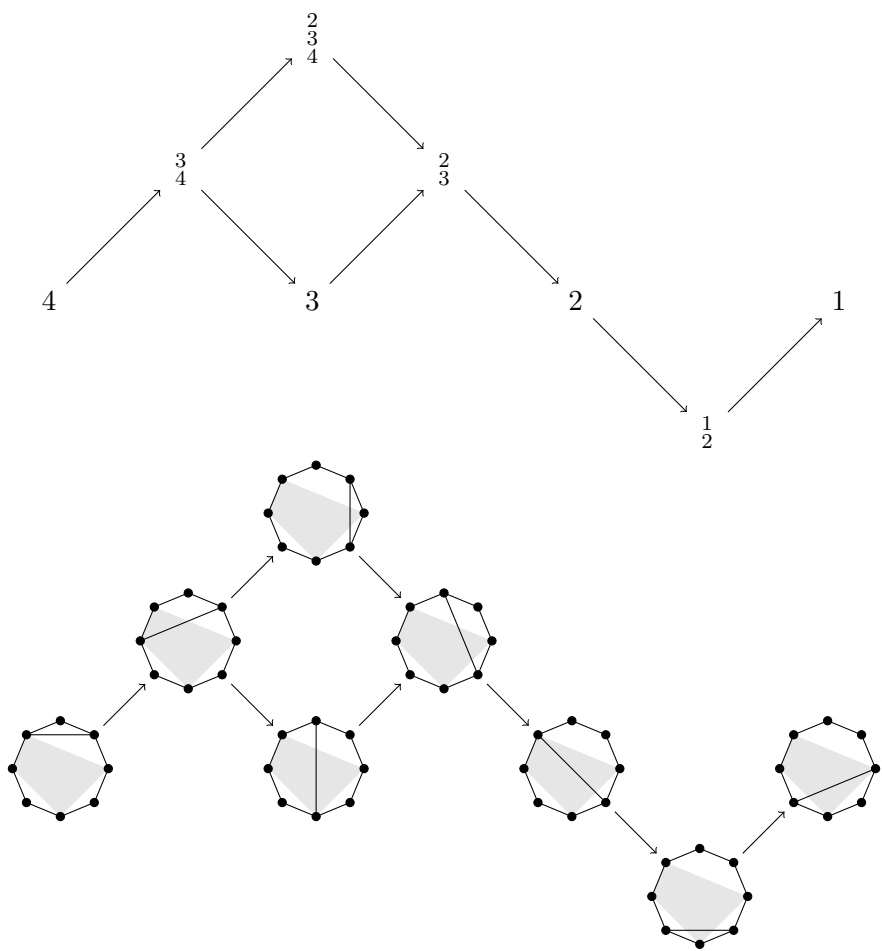
\begin{figure}
\begin{tikzpicture}[pin distance=.75,scale=1.75]
\foreach \x/\y/\d [count=\n] in 
	{1/1/4, 
	 2/2/\dimv34, 
	 3/1/3,3/3/\dimvv234,
	 4/2/\dimv23,
	 5/1/2,
	 6/0/\dimv12,
	 7/1/1
	 }{
	\node[name=\n] at (\x,\y) {$\d$};
}
\path[->] (1) edge (2) 
	       (2) edge (3) edge (4)
	       (3) edge (5) 
	       (4) edge (5)
	       (5) edge (6)
	       (6) edge (7)
	       (7) edge (8);
\end{tikzpicture}

\begin{tikzpicture}[scale=.85]
{[]
\node[name=n1,outer sep=1cm] at (0,0) {};
\foreach \a [count=\n] in {0,1,2,3,4,5,6,7}
	\node[name=\n,solid] at (360*\a/8:.75cm) {};
\foreach \x [remember=\x as \lastx (initially 8)] in {1,...,8}
	\draw (\lastx.center) -- (\x.center);
\draw (2) -- (4);
{[on background layer] \fill[black!10] (1.center) -- (4.center) --(5.center) -- (7.center) -- cycle;} 
}
{[xshift=2cm,yshift=2cm]
\node[name=n2,outer sep=1cm] at (0,0) {};
\foreach \a [count=\n] in {0,1,2,3,4,5,6,7}
	\node[name=\n,solid] at (360*\a/8:.75cm) {};
\foreach \x [remember=\x as \lastx (initially 8)] in {1,...,8}
	\draw (\lastx.center) -- (\x.center);
\draw (2) -- (5);
{[on background layer] \fill[black!10] (1.center) -- (4.center) --(5.center) -- (7.center) -- cycle;} 
}
{[xshift=4cm,yshift=0cm]
\node[name=n3,outer sep=1cm] at (0,0) {};
\foreach \a [count=\n] in {0,1,2,3,4,5,6,7}
	\node[name=\n,solid] at (360*\a/8:.75cm) {};
\foreach \x [remember=\x as \lastx (initially 8)] in {1,...,8}
	\draw (\lastx.center) -- (\x.center);
\draw (3) -- (7);
{[on background layer] \fill[black!10] (1.center) -- (4.center) --(5.center) -- (7.center) -- cycle;} 
}
{[xshift=4cm,yshift=4cm]
\node[name=n4,outer sep=1cm] at (0,0) {};
\foreach \a [count=\n] in {0,1,2,3,4,5,6,7}
	\node[name=\n,solid] at (360*\a/8:.75cm) {};
\foreach \x [remember=\x as \lastx (initially 8)] in {1,...,8}
	\draw (\lastx.center) -- (\x.center);
\draw (2) -- (8);
{[on background layer] \fill[black!10] (1.center) -- (4.center) --(5.center) -- (7.center) -- cycle;} 
}
{[xshift=6cm,yshift=2cm]
\node[name=n5,outer sep=1cm] at (0,0) {};
\foreach \a [count=\n] in {0,1,2,3,4,5,6,7}
	\node[name=\n,solid] at (360*\a/8:.75cm) {};
\foreach \x [remember=\x as \lastx (initially 8)] in {1,...,8}
	\draw (\lastx.center) -- (\x.center);
\draw (3) -- (8);
{[on background layer] \fill[black!10] (1.center) -- (4.center) --(5.center) -- (7.center) -- cycle;} 
}
{[xshift=8cm,yshift=0cm]
\node[name=n6,outer sep=1cm] at (0,0) {};
\foreach \a [count=\n] in {0,1,2,3,4,5,6,7}
	\node[name=\n,solid] at (360*\a/8:.75cm) {};
\foreach \x [remember=\x as \lastx (initially 8)] in {1,...,8}
	\draw (\lastx.center) -- (\x.center);
\draw (4) -- (8);
{[on background layer] \fill[black!10] (1.center) -- (4.center) --(5.center) -- (7.center) -- cycle;} 
}
{[xshift=10cm,yshift=-2cm]
\node[name=n7,outer sep=1cm] at (0,0) {};
\foreach \a [count=\n] in {0,1,2,3,4,5,6,7}
	\node[name=\n,solid] at (360*\a/8:.75cm) {};
\foreach \x [remember=\x as \lastx (initially 8)] in {1,...,8}
	\draw (\lastx.center) -- (\x.center);
\draw (6) -- (8);
{[on background layer] \fill[black!10] (1.center) -- (4.center) --(5.center) -- (7.center) -- cycle;} 
}
{[xshift=12cm,yshift=0cm]
\node[name=n8,outer sep=1cm] at (0,0) {};
\foreach \a [count=\n] in {0,1,2,3,4,5,6,7}
	\node[name=\n,solid] at (360*\a/8:.75cm) {};
\foreach \x [remember=\x as \lastx (initially 8)] in {1,...,8}
	\draw (\lastx.center) -- (\x.center);
\draw (6) -- (1);
{[on background layer] \fill[black!10] (1.center) -- (4.center) --(5.center) -- (7.center) -- cycle;} 
}

\path[<-] (n1) edge (n2) 
	       (n2) edge (n3) edge (n4)
	       (n3) edge (n5) 
	       (n4) edge (n5)
	       (n5) edge (n6)
	       (n6) edge (n7)
	       (n7) edge (n8);
\end{tikzpicture}
\caption{The Auslander-Reiten quiver of a tilted algebra of type $\AA_4$}\label{fig tilted}
\end{figure}
\end{ex}

\bigskip

\begin{minipage}[pos]{7in}{\small  Lucas David-Roesler
\\
Department of Mathematics\\
University of Connecticut\\
lucas.roesler@uconn.edu\\
\\
Ralf Schiffler
\\
Department of Mathematics\\
University of Connecticut\\
schiffler@math.uconn.edu\\}  
\end{minipage}
\end{document}